\documentclass[a4paper]{article}
\usepackage[affil-it]{authblk}
\usepackage{amsfonts}
\usepackage{amsmath}
\usepackage{amsthm}
\usepackage{amssymb}
\usepackage{hyperref}
\usepackage{autonum}
\usepackage[all,cmtip]{xy}
\usepackage[utf8]{inputenc}
\usepackage[T1]{fontenc}
\usepackage[title]{appendix}
\usepackage[usenames,dvipsnames]{color}
\usepackage{tikz}
\usepackage{tikz-cd}
\usepackage[left=2.5cm,top=3cm,right=2.5cm,bottom=3cm,bindingoffset=0.5cm]{geometry}
\usepackage[outdir=./]{epstopdf}

\hypersetup{%
  colorlinks=false,
  linkbordercolor=red,
  pdfborderstyle={/S/U/W 0.3}
}

\newtheorem{theorem}{Theorem}
\newtheorem*{theorem*}{Theorem}
\newtheorem{corollary}[theorem]{Corollary}
\newtheorem{proposition}[theorem]{Proposition}
\newtheorem{lemma}[theorem]{Lemma}
\newtheorem{definition}[theorem]{Definition}

\newtheorem{example}[theorem]{Example}

\newtheorem{remark}[theorem]{Remark}

\newcommand{\fS}{\mathfrak{S}}

\DeclareMathOperator{\Ad}{Ad}

\DeclareMathOperator{\lcm}{lcm}
\DeclareMathOperator{\Tr}{tr}
\DeclareMathOperator{\N}{N}

\DeclareMathOperator{\cov}{covol}
\DeclareMathOperator{\vol}{vol}

\DeclareMathOperator{\opp}{op}

\DeclareMathOperator{\Gal}{Gal}
\DeclareMathOperator{\ord}{ord}

\begin{document}
\title{Lattice packings through division algebras}
\author{Nihar Prakash Gargava%
\thanks{Electronic address: \texttt{nihar.gargava [AT] epfl.ch}}}
\affil{Chair of Number Theory,\\ Department of Mathematics,\\ École Polytechnique Fédérale de Lausanne}

\maketitle

\abstract{
  In this text, we will show the existence of lattice packings in a family of dimensions by employing division algebras. This construction is a generalization of Venkatesh's lattice packing result \cite{AV}. In our construction, we replace the appearance of the cyclotomic number field with a division algebra over the rational field.  We employ a probabilistic argument to show the existence of lattices in certain dimensions with good packing densities.
  The approach improves the best known lower bounds on the lattice packing problem for certain dimensions. 

We work with a moduli space of lattices that are invariant under the action of a finite group, one that can be embedded inside a division algebra. To obtain our existence result, we prove a division algebra variant of the Siegel's mean value theorem. In order to establish this, we describe a useful description of the Haar measure on our moduli space and a coarse fundamental domain to perform the integration.
}

\ 

\tableofcontents
\section*{Introduction}
Let $V$ be a real vector space of $d$ dimensions with a given inner product $\langle \ , \ \rangle$. When we say $\Lambda$ is a lattice in $V$, we mean that $\Lambda$ is a discrete closed subgroup $\Lambda \subseteq V$ such that $V / \Lambda$ has a finite volume from the induced measure. The volume of $V / \Lambda$ is also called the covolume of $\Lambda$.

Given a lattice $\Lambda$, take a real number $r > 0$ and consider the collection of open balls $\{ B_{r}\left( v \right)\}_{v \in \Lambda}$. Such a collection of balls are said to be a sphere packing in $V$ if no non-trivial pairs of these balls intersect. That is, for any $v_1, v_2 \in \Lambda$, $B_{r}(v_1) \cap B_{r}(v_2) \neq \emptyset \Rightarrow v_1 = v_2$.
Such an arrangement is called a lattice sphere packing, or simply lattice packing.

We have a notion of the packing efficiency of a lattice packing defined as 
\begin{align}
   \lim_{R \rightarrow \infty} \frac{\mu\left( B_{R}(0) \cap \left(  \bigsqcup_{v \in \Lambda} B_{r}(v) \right) \right)}{\mu\left( B_{R}(0) \right)} = 
 \frac{\mu( B_{r}(0) )}{ \mu(V / \Lambda)}.
\end{align}
where $\mu$ is the Lebesgue measure on $V$ induced by the inner product.
This is always a real number in the open interval $(0,1]$.

We define the $d$-dimensional sphere packing constant as
\begin{align}
  c_{d} = \sup\left\{ \mu\left( g B_{r}(0)\right) \ | \ r> 0,~g \in SL(V) \text{ and } g B_{r}(0) \cap \Lambda_0 = \{ 0\}\right\},
\end{align}
where $SL(V)$ is the group of unimodular linear transformation on $V$ and $\Lambda_0$ is any unit covolume lattice in $V$. 
It then follows that the tightest possible lattice packing in $V$ has a packing density equal to $2^{-d} c_d$. Indeed, if we double the radius of the origin-centered ball in the packing, then it can contain no non-zero centers of the other balls and translating the centers of the balls to $\Lambda_0$ with an appropriate $g \in SL(V)$ shows what is required.

The exact value of $c_d$ is known only for $d \in \{ 1,2,3,4,5,6,7,8,24\}$ (\cite{CS}, \cite{CK2009}). Bounds exist for other values.
There are several known results that establish lower bounds on $c_d$ for various class of dimensions. Some of the celebrated results are compiled in Table \ref{tab:asymptots}.

\begin{table}[htbp]
  \centering
  \begin{tabular}{|c|c|c|}
    \hline
    Lower bound on packing density & Contributed by  & Dimensions covered \\
    \hline 
    $c_{d} \ge 1$ & Minkowski (cf. \cite{h43}) & Any $d \ge 1$ \\
    $c_{d} \ge 2(d-1)$ & Ball \cite{ba92} & Any $d \ge 1$ \\
    $c_{4n} \ge 8.8 n$ & Vance \cite{va2011}  &  $d = 4n, n \ge 1$ \\
    $c_{2 \varphi(n)} \ge n$ & Venkatesh \cite{AV} & $d= 2\varphi(k)$, for some $k \ge 1$ \\ 
    \hline 
  \end{tabular}
  \caption{Available lower bounds in large dimensions}
  \label{tab:asymptots}
\end{table}

The last result due to Venkatesh is the best known lower bound asymptotically. Note that suppose that we write $n=  2 \varphi(n) ( \tfrac{n}{2 \varphi(n)} )$. Now from Mertens' theorem, we know that $\tfrac{n}{2 \varphi(n)}$ can be as big as $O(\log \log n) = O(\log \log d)$. This happens for the subsequence of dimensions $d = 2\varphi(k)$ where $k=p_1 p_2 \dots p_k$ where $\{ p_1, p_2,\dots\}$ are prime numbers indexed increasingly. Hence, along a sequence of dimensions, the lower bound due to Venkatesh is better than any linear bound.

Theorem \ref{pr:boundbetter} is the main result in this text, restated below in a convenient form.

\begin{theorem*}
  Let $D$ be a finite-dimensional division algebra over $\mathbb{Q}$. Let $\mathcal{O} \subseteq D$ be an order (see Definition \ref{de:order}) and $G_{0} \subseteq \mathcal{O}$ be a finite group embedded in the multiplicative group of $D$. Then if $d=2\dim_{\mathbb{Q}}D$, then
  \begin{align}
    c_{d} \ge \# G_{0}.
  \end{align}
\end{theorem*}

Since a number field is also a division algebra over $\mathbb{Q}$, we recover the result of Venkatesh by setting $D = \mathbb{Q}(\mu_n)$, the $n$th cyclotomic field, $\mathcal{O}$ to be the ring of integers in $\mathbb{Q}(\mu_{n})$ and the $n$th cyclotomic field and $G_0 = \langle \mu_{n}\rangle$. Hence, Venkatesh's construction can be recovered from this theorem.

Figure \ref{fig:newresult} directly compares the previously existing set of lower bounds with newer results obtained. Note that although most of the points in the plot are from Venkatesh's result, the bounds obtained from division algebras are slightly better wherever they apply.

To get packing bounds from finite subgroups embedded in division algebras, we exploit Amitsur's classification results from \cite{Amit55} which contains a description of every possible type of finite group $G_0$ that can be used for obtaining lower bounds. The reader can find a summary of this classification result in Theorem \ref{th:amit_div} from Section \ref{se:improv}. 

One of the two infinite families of groups mentioned in the classification leads to the following sequence of dimensions mentioned in Theorem \ref{th:improv}.

\begin{theorem*}
  There exists a sequence of dimensions $\{ d_i\}_{i=1}^{\infty}$ such that for some $C>0$, we have $c_{d_{i}} > C d_{i} (\log \log d_{i})^{\frac{7}{24}}$
  and the lattices that achieve this bound in each dimension are symmetric under the linear action of a non-commutative finite group.
\end{theorem*}

\begin{figure}[!h]
  \centering
\input{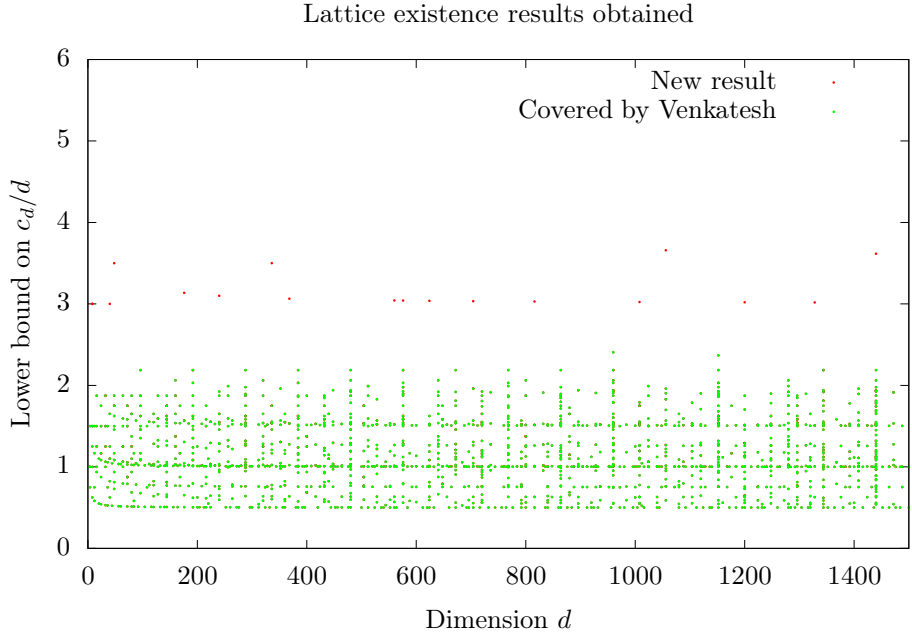}
\caption{The points shown in this figure are the points $(2\dim_{\mathbb{Q}} D, \tfrac{|G_0|}{2 \dim_{\mathbb{Q}}D})$ as $D$ and $G_0$ respectively varies across division algebras and finite groups mentioned in Theorem \ref{th:amit_div}. When the division algebra $D$ is a cyclotomic field over $\mathbb{Q}$, this corresponds to Venkatesh's result.}
    \label{fig:newresult}
\end{figure}

\begin{figure}[!h]
  \centering
\input{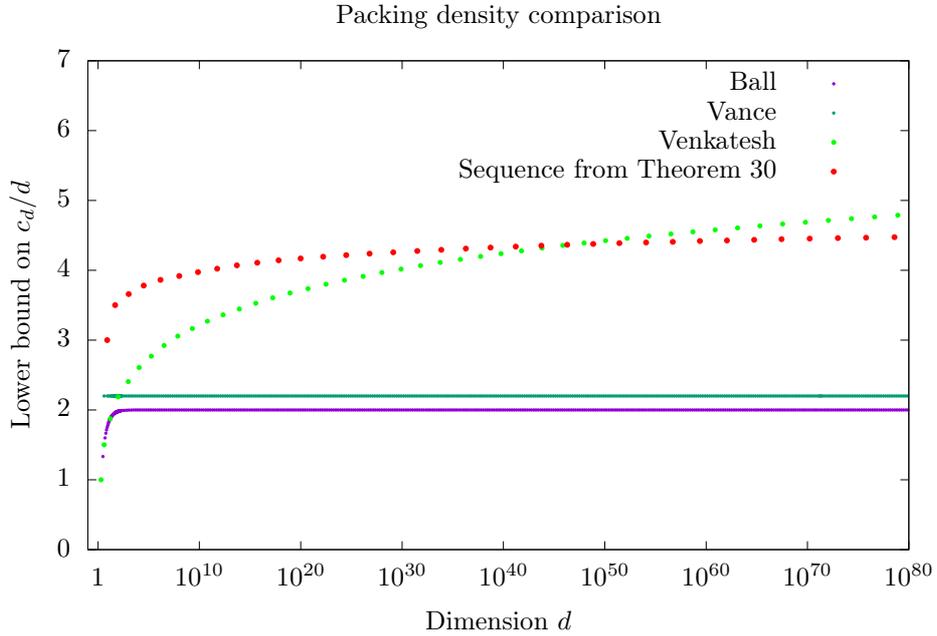}
\caption{The sequence of Venkatesh is better after $d \sim 1.98 \times 10^{46}$ than the sequence obtained from Theorem \ref{th:improv} and outperforms any linear bound on $c_d$ since it grows at $O(d (\log \log d)^{\tfrac{7}{24}})$.}
    \label{fig:graph}
\end{figure}

The significance of this result is that it yields improvements on lower bounds on $c_d$ in a collection of dimensions $d$ for $d \le 1.98 \times 10^{46}$ (see Figure \ref{fig:graph}). Nonetheless, because of the $7/24$ in the exponent of $\log \log d$, the asymptotic growth does not keep up with Venkatesh's growth of $O(d \log \log d)$. This $7/24$ appears because of the density of primes modulo which multiplicative order of 2 is odd. This restriction is imposed because Amitsur's classification result. See Section \ref{se:improv} for this interesting discussion.

The main technique to achieve the lower bounds here is to establish a division algebra variant of Siegel's mean value theorem \cite{Sie45}. This is the same probabilistic technique that makes the result of \cite{AV} possible. The key motivation of this theorem is to average a lattice-sum function on a collection of lattices that have some prescribed symmetries. It comes up as Theorem \ref{th:siegel} in our text.
\begin{theorem*}

  Let $D$ be a $\mathbb{Q}$-division algebra containing an order $\mathcal{O} \subseteq D$. Let $D_{\mathbb{R}} = D \otimes_{\mathbb{Q}} \mathbb{R}$ and $G = SL_k(D_{\mathbb{R}})$ and $\Gamma = SL_k(\mathcal{O})$, for some $k \ge 2$. Let $dg$ be the probability measure on $G/\Gamma$ that is left-invariant under $G$ action. Then for any $f \in C_{c} (D_{\mathbb{R}}^{k})$, we obtain that 
  \begin{align}
    \int_{G / \Gamma}\left( \sum_{v \in g \mathcal{O}^{k} \setminus \{ 0\}} f(v) \right) dg =  \int_{D_{\mathbb{R}}^{k}}^{} f(x) dx,
  \end{align}
  where $dx$ is a Lebesgue measure on $D_{\mathbb{R}}^{k}$ with respect to which $\mathcal{O}^{k}$ has a covolume of $1$.
\end{theorem*}

In order to establish this mean value theorem, most of the effort is directed towards finding a suitable ``coarse'' fundamental domain of $G/\Gamma$ for integrating the left-hand side, which is done in Section \ref{se:reduction_theory}. The treatment of fundamental domains here follows that of Weil \cite{W58}, in which Weil covers the case of constructing arithmetic subgroups using real semisimple algebras with an involution and generalizes the construction of Siegel domains to the case when that algebra is the tensor product of a $\mathbb{Q}$-division algebra with $\mathbb{R}$. The work was eventually vastly generalized by Borel and Harish-Chandra in \cite{BHC62} to create a much more general theory of Siegel domains, but we will use the following more elementary approach of Weil. 

After establishing the ``coarse'' fundamental domain of $G/\Gamma$, the proof the mean value theorem can be found in Section \ref{se:SMT}.

In Section \ref{se:improv}, we also give some some sequences of dimensions in which we can achieve the same $O(d \log \log d)$ asymptotic growth using non-commutative finite groups. This is mentioned in Proposition \ref{pr:asymptotic}.

Apart from the given sequences, Figure \ref{fig:newresult} suggests that there may be lots of (possibly infinitely many) improvements on lower bounds on $c_d$ for individual dimensions $d$ that can be shown using the given methodology. It remains a questions of finding good ways to systematically generate such dimensions.

\section{Matrices over division algebras}

The goal of this section is to guide the reader towards the division algebra version of Siegel's mean value theorem. To get an overview of the theory of matrices over real semisimple algebra, see Appendix \ref{se:mat_over_real_ss_algebra}.

\subsection{Reduction theory of matrices over division algebras}\label{se:reduction_theory}

For a positive definite symmetric quadratic form $q: \mathbb{R}^{n} \rightarrow \mathbb{R}$, what is the set $\{ q(x)\}_{x \in \mathbb{Z}^{n} \setminus \{ 0\}}$? There is an enormous amount of literature and decades of mathematical research around this question.
But one important step before proceeding anywhere is to realize $g \in GL_n(\mathbb{Z})$, $\{ q(x)\}_{x \in \mathbb{Z}^{n} \setminus \{ 0\}} = \{ q( g(x) )\}_{x \in \mathbb{Z}^{n} \setminus \{ 0\}}$. Hence $q$ and $q \circ g$ are essentially the same quadratic forms as far as their values on integral points are concerned.

Reduction theory of quadratic forms generally refers to attempts at finding some suitable representative of a quadratic form modulo this equivalence. In this section, we will generalize the classical Minkowski-Siegel reduction theory of quadratic forms to the case of the types of quadratic forms we have talked about so far. To do so, we will first reframe the notion of ``integral points'' accordingly.

\begin{definition}\label{de:order}
  Let $A_{\mathbb{Q}}$ be a semisimple $\mathbb{Q}$-algebra. Then an additive subgroup $\mathcal{O} \subseteq A_{\mathbb{Q}}$ is called an order of $A$ if
  \begin{itemize}
    \item It is a finitely generated $\mathbb{Z}$-module.
    \item $\mathbb{Q}{\otimes_\mathbb{Z}} \mathcal{O}  = A_{\mathbb{Q}}$.
    \item It is closed under multiplication, that is $a , b \in \mathcal{O} \Rightarrow ab \in \mathcal{O}$. 
    \item $1_{A} \in \mathcal{O}$.
  \end{itemize}
\end{definition}

\begin{example}
  $\mathbb{Z} \subset \mathbb{Q}$ is an order. In general, for any number field $K$, the ring of integers $\mathcal{O}_{K}$ is an order.
\end{example}

When $\mathcal{O} \subseteq A_{\mathbb{Q}}$ is an order, $M_k(\mathcal{O})$ is an order within $M_{k}(A_{\mathbb{Q}})$. Moreover, $\mathcal{O} \subset A_{\mathbb{R}} = A_{\mathbb{Q}} \otimes \mathbb{R}$ is a lattice in the Euclidean topology. We will often refer to $\mathcal{O}$ as the ``integral points of $A$'' and as elements of $M_k(\mathcal{O})$ as ``integral matrices'' in $M_k(A)$.

\begin{remark}\label{re:integral}
  This notion of ``integral matrices'' can be reconciled with common sense in the following way. 
  Since $\mathcal{O}$ spans $A_{\mathbb{Q}}$, we can make a $\mathbb{Q}$-basis of $A_{\mathbb{Q}}$ from elements of $\mathcal{O}$. Extending this basis to a basis of $A_{\mathbb{Q}}^{k}$, we can recognize the algebra $M_{k}(A_{\mathbb{Q}})$ as an algebra of real matrices acting on $A^{k}_{\mathbb{Q}}$. Under this identification, the elements of $M_k(\mathcal{O})$ are exactly those elements of $M_{k}(A)$ whose entries as rational matrices are integers.

  Making this more precise, denote $d = \dim_{\mathbb{Q}} A_{\mathbb{Q}}$. Then there exists a faithful $\mathbb{R}$-algebra morphism $\pi:M_k(A_{\mathbb{R}}) \rightarrow M_{kd}(\mathbb{R})$ that maps $M_k(\mathcal{O})$ inside $M_{kd}(\mathbb{Z})$. In fact we see that, $M_k(\mathcal{O}) = \pi^{-1}(M_{kd}(\mathbb{Z}))$, because if $\pi(m) \in M_{kd}(\mathbb{Z})$, $m e_{i} \in \mathcal{O}^{k}$, when $e_{i} =(0,\dots,0,1_{A},0,\dots,0)\in \mathcal{O}^{k}$. 
\end{remark}

From now on, we will restrict our setting to the following. Instead of talking about a general semisimple $\mathbb{R}$-algebra $A$, we will talk of when $A$ is of the form\footnotemark $D_{\mathbb{R}} = D \otimes_{\mathbb{Q}} \mathbb{R}$ for some $\mathbb{Q}$-division algebra $D$. We will now also fix an order $\mathcal{O} \subseteq D \subseteq D_{\mathbb{R}}$ and this will be the ``integral points'' of $D_{\mathbb{R}}$. We will fix on $D_{\mathbb{R}}$ a positive involution $(\ )^{*}: D_{\mathbb{R}} \rightarrow D_{\mathbb{R}}$ (see Definition \ref{de:pos_inv}, Appendix \ref{se:mat_over_real_ss_algebra}).

\footnotetext{Why is $D_{\mathbb{R}}$ semisimple? The trace form $(a,b)\mapsto \Tr(ab)$ is clearly non-degenerate on $D$. It is classically known that the trace form on a finite-dimensional $k$-algebra is non-degenerate if and only if it is absolutely semisimple, i.e. $A \otimes_{k} L$ is semisimple for any field extension $L$ of $k$. }

The following theorem is a generalization of the classical Minkowski-Siegel reduction theorem, and is mentioned by the same name in \cite{W58}.
\begin{theorem}\label{th:MS}
  For the setting $\mathcal{O}\subseteq D \subseteq D_{\mathbb{R}}$ above, there exist constants $C_1,C_2, C_3 > 0$ and a relatively compact set $\omega_0 \subseteq D_{\mathbb{R}}$ depending only on $\mathcal{O}, D_{\mathbb{R}}$ and $k$ such that whenever there exists a positive-definite symmetric element $a \in M_{k}(D_{\mathbb{R}})$, there exists an $m \in M_k(\mathcal{O})$ such that the following conditions are met.
  \begin{enumerate}
    \item $|\N(m)|< C_3$\label{cond:norm}
    \item The Cholesky decomposition (see Theorem \ref{th:cholesky}) of $ m^{*} a m  = t^{*} d t$ satisfies
      \begin{enumerate}
	\item $ \frac{ d_{ij} } { \Tr(d_{ij}) } $ lies in $\omega_0$.\label{cond:a}
	\item $\Tr(d_{ii}) \le C_{1} \Tr(d_{(i+1)(i+1)}) $.\label{cond:b}
	\item $\Tr(t_{ij}^{*} t_{ij}) \le C_{2}$.\label{cond:c}
      \end{enumerate}
  \end{enumerate}
\end{theorem}
\begin{proof}

  See \cite[Theorem 2]{W58}.

\end{proof}
\begin{remark}\label{re:normOmegaoaway}
  The set $\omega_{0}$ can be assumed to be inside $\{ d \in D_{\mathbb{R}} \ | \ \Tr(d) = 1\} $. This is because $\Tr(d_{ii}/\Tr(d_{ii})) =1$.
Furthermore, $\omega_{0}$ can be chosen to be relatively compact inside $D_{\mathbb{R}}^{*}$, the invertible elements of $D_{\mathbb{R}}$. In particular, this means that $\{N_{D_{\mathbb{R}}}(x)\}_{x \in \omega_0}$ is bounded away from $0$. 


\end{remark}
We can reformulate the above using the definition of a Siegel domain. Given a relatively compact set 
$\omega_{0} \subseteq D_{\mathbb{R}}$ and two constants $C_1, C_2 > 0$, then we define a Siegel domain 
\begin{align}\label{eq:siegel_definition}
  \mathfrak{S} = \mathfrak{S}_{\omega_{0} , C_1 ,C_2} = & \{ a \in M_k(D_{\mathbb{R}})  \ | \ a \text{ is symmetric positive definite } \\  & \text{ whose Cholesky decomposition $a = t^{*} d t$ satisfies } \\ &  \text{ conditions (a), (b) and (c) of Theorem~\ref{th:MS}}  \}.
\end{align}

In this context, what Theorem~\ref{th:MS} tells us is that there exists a Siegel domain $\mathfrak{S}$ such that, for any positive-definite symmetric $a \in M_k(D_{\mathbb{R}})$ an integral matrix $m$ of bounded norm can make $m^{*} a m \in \mathfrak{S}$. 

However, we can do a small correction to replace $m$ with $m'b$, where $m'$ is such that $\N(m')=1$ and $b$ is among finitely many candidates in $M_k(\mathcal{O})$. 
This will be used in Lemma \ref{le:surjects_from_fund_domain}, for example.

\begin{lemma}\label{le:normonecorrection}
  Given a constant $C>1$, we can find finitely many elements $b_1, b_2, b_3 ,\dots , b_m \in M_k(D)$ such that any $b \in M_k(\mathcal{O})$ with $1 \le | \N(b)|  \le C$ can be written as $b = b' b_i$ for some $i$, with $\N(b') = 1$.A
\end{lemma}
\begin{proof}

  See \cite[Lemma 10.2]{W58}.
\end{proof}



\subsection{Group of unit norm matrices}

\label{se:lie}
This subsection is going to set up the measure-theoretic requirements for Theorem \ref{th:siegel}. We will work in the homogeneous space that is the quotient of the following two groups.

\begin{align}
  G = &  \{  a \in M_{k}(D_{\mathbb{R}}) \ | \ \N(a) = 1\},\\
  \Gamma = &  \{  a \in M_{k}(\mathcal{O}) \ | \  \N(a) = 1\} .
\end{align}

Clearly, $G$ is a group. Why is $\Gamma$ a group? To see that it is a group, one must realize the matrices in $\Gamma$ as integral matrices in the sense of Remark \ref{re:integral}. Then, the group $\Gamma$ is just the subgroup of determinant $1$ integral matrices in $G$. Furthermore, this also shows that $\Gamma \subseteq G$ is a discrete group.

\begin{remark}
  Alternatively, it is also possible to write $G$ as $SL_k(D_{\mathbb{R}})$ and $\Gamma$ as $SL_k(\mathcal{O})$. We will also use the notation $SL_k(D)$ to mean the unit norm matrices of $GL_k(D)$.
\end{remark}

We want to describe a Haar measure on $G$. For that, we will use the following analogue of the Iwasawa decomposition.
\begin{align}
  GL_k(D_{\mathbb{R}})  & = \{ g \in M_k(D_{\mathbb{R}}) \ | \ g \text{ is not a zero divisor }  \} , \\
  K &  = \{ \kappa \in G \ | \ \kappa ^{*} \kappa = 1_{M_k(A)}, \N(\kappa) = 1\}, \\
  A_0 & = \{ a \in G \ | \ a \text{ is diagonal}, a_{ii} \text{ invertible}, \N(a_{ii}) >0 \},\\
  N & =  \{ n \in G \ | \ n \text{ is upper triangular with $1_{A}$ on the diagonal entries}\}.
\end{align}
Topologically, $G$ is a Lie group and the groups $K,A_{0},N$ are also Lie group topologies as closed subgroups of $GL_k(D_{\mathbb{R}})$. Note that $A_{0} \subseteq G$, so $a \in A_{0} \Rightarrow \N(a) =1 $. See Proposition \ref{pr:iwasawa} in Appendix \ref{se:haar_on_G} for a variant of Iwasawa decomposition for the group $G$.

We will now describe a Haar measure on $G$. For any topological space $X$, we will denote the vector space of compactly supported continuous $\mathbb{R}$-functions on $X$ as $C_{c}(X)$.

\begin{proposition}\label{pr:Gisunimodular}
  Let $d\kappa, da, dn$ be Haar measures on $K, A_{0}, N$ respectively. Then, the following is a Haar measure on $G$.
  \begin{align}
    C_{c}(G) \rightarrow &\  \mathbb{R} \\ 
    f \mapsto &  \  \int_{N} \int_{A_{0}} \int_{K}f(\kappa a n)\left( \prod_{i<j}^{} \frac{| \N(a_{ii})|}{ | \N(a_{jj}) |} \right)  d\kappa da dn 
  \end{align}
\end{proposition}

\begin{proof}
  See Appendix \ref{se:haar_on_G}.
\end{proof}

  Let $D_{\mathbb{R}}^{(1)}$ denote the kernel of $\N: D_{\mathbb{R}}^{*} \rightarrow \mathbb{R}$. In other words $D_{\mathbb{R}}^{(1)}$ is the set of unit norm elements of $D_{\mathbb{R}}$. Note that the group $A_{0}$ can be further decomposed as $A_0 = A^{(1)}A^{\mathbb{R}}$ where 
  \begin{align}
    A^{(1)} & = \{ a \in G \ | \ i\neq j \Rightarrow a_{ij}=0, \N(a_{ii})=1 \},\\
    A^{\mathbb{R}}  & = \{ a' \in G \ | \ i\neq j \Rightarrow a_{ij}'=0, a'_{ii} \in  \mathbb{R}_{>0} \subseteq D_{\mathbb{R}}  \}.
  \end{align}
  Note that $A^{\mathbb{R}} \cap A^{(1)} = \{ 1_{D}\}$. This decomposition is simply a consequence of writing $a_{ii} = N(a_{i})^{1/d}  \left( { a_{ii}}{N(a_{ii})^{-1/d} }\right)$, where $d= [D_{\mathbb{R}}: \mathbb{R}]$ so that $a_{ii}N(a_{ii})^{-1/d}$ is of norm one. 

  \begin{remark}
    The group $A^{\mathbb{R}}$ is actually the identity component of a maximal $\mathbb{Q}$-torus of $G$. In Chapter 18.5 of \cite{Mo2001}, the $\mathbb{Q}$-rank of $SL_k(D)$ is mentioned as $k-1$, which is exactly the rank of this torus.
  \end{remark}

  \begin{corollary}\label{co:haar}
    Let $d\kappa$, $da'$, $da$, $dn$ be Haar measures on $K,A^{\mathbb{R}},A^{(1)}, N$ respectively. Then, the following is a Haar measure on $G$.
    \begin{align}
    C_{c}(G) \rightarrow &\  \mathbb{R} \\ 
    f \mapsto &  \  \int_{N} \int_{A^{(1)}} \int_{A^{\mathbb{R}}} \int_{K}f(\kappa a' a n)\left( \prod_{i<j}^{}  \frac{ a'_{ii}}{  a'_{jj} }  \right)^{d}  d\kappa da' da dn 
  \end{align}
  \end{corollary}

\begin{remark}
  It should be possible to generalize this treatment of Haar measure to the setting of a general semisimple algebra $A$ instead of $D_{\mathbb{R}}$ by meaningfully defining groups like $GL_k(A)$, $SL_k(A)$ and so on.
\end{remark}

\subsection{Siegel's mean value thorem}\label{se:SMT}
Observe that $\mathcal{O}^{k} \subseteq D_{\mathbb{R}}^{k}$ is a lattice that remains invariant under the action of elements of the group $\Gamma \subseteq G$. Hence, we can make the following identification of topological measure spaces.
\begin{align}
  G/\Gamma \simeq  \{  g \mathcal{O}^{k}  \ | \ g \in G\}.
\end{align}
We will shortly show that this measure space has a finite measure. Furthermore, we will state a nice averaging result about the expected value of a lattice-sum function over lattices in this space.

To begin, we will now make a more useful version of a Siegel domain $\fS^{*} \subset G$, one that we can fit inside $G$ and and such that $\fS^{*} \Gamma = G$. This shall be a Siegel domain of matrices, whereas the previous definition $\fS$ was a Siegel domain of quadratic forms. Let $\omega_{1}\subseteq D_{\mathbb{R}}^{(1)}$ be a relatively compact set and let $c_1,c_2>0$. Also, let $b_1, b_2, \dots , b_m$ be some elements of $GL_k(D) \subseteq G$
\begin{align}
  \underline{A}^{\mathbb{R}} & = \{ a \in GL_k(D_{\mathbb{R}})  \ | \ a'_{ij} = 0 \text{ for }i\neq j,a_{ii}' \in \mathbb{R}_{>0} \subset D_{\mathbb{R}} \}, \\
  A_{\omega_1}^{(1)} & =  \{ a \in A^{(1)}\ | \ a_{ii} \in \omega_{1}  \} ,\\
  A_{c_1}^{\mathbb{R}} & =  \{ a' \in A^{\mathbb{R}} \ | \ a'_{ii} \in \mathbb{R}_{>0} \subseteq D_{\mathbb{R}} ,  a'_{ii} \le c_1 a'_{i+1,i+1}  \} ,\\
  \underline{A}_{c_1}^{\mathbb{R}} & = \{ a \in \underline{A}^{\mathbb{R}} \ | \ a_{ii}' \in \mathbb{R}_{>0} \subseteq D_{\mathbb{R}}, a'_{ii} \le c_1 a'_{i+1,i+1} \},\\
  N_{c_2} & =  \{ n \in G \ | \ n \text{ is upper triangular with $1_{D}$ on diagonals }, \Tr(n_{ij}^{*}n_{ij})  < c_{2}  \},\\
  \fS^{1}& = \fS^{1}_{\omega_{1},c_1,c_2}  =  K A_{\omega_1}^{(1)} \underline{A}_{c_1}^{\mathbb{R}} N_{c_2},\\
  \fS^{*}& = \fS^{*}_{\omega_{1},c_1,c_2}  =  \left( \bigcup_{i=1}^{m} \fS^{1} b_i^{-1} \right) \cap G  =  \bigcup_{i=1}^{m} \N(b_i)^{1/dk}  (K A^{(1)}_{\omega_1} A_{c_1}^{\mathbb{R}} N_{c_{2}}) b_i^{-1}.
\end{align}

We can now relate this to the previously discussed generalization of Minkowski-Siegel (Theorem \ref{th:MS}).

\begin{lemma}\label{le:surjects_from_fund_domain}
  For some choice of $\omega_1,c_1,c_2$ in the definition above and for some choice of $b_1, b_2 , \dots , b_m \in GL_k(D)$, we have that the construction $\fS^{*} \subseteq G$ above satisfies $ \fS^{*} \Gamma  = G$. In other words, $\fS^{*} \subseteq G$ surjects via the map $G \rightarrow G/\Gamma$.
\end{lemma}
\begin{proof}
  What we want to really show is that for some choice of $\fS^{*}$, for every $g \in G$, there will exist a $b \in \Gamma$ such that $gb \in \fS^{*}$.
  
  Let $\fS = \fS_{\omega_0,C_1,C_2} \subset M_k(D_{\mathbb{R}})$ be the set defined in Equation (\ref{eq:siegel_definition}), where $\omega_{0},C_1,C_2$ are chosen such that they satisfy conditions of Theorem \ref{th:MS} for the given choice of $D_{\mathbb{R}}, \mathcal{O}$ and $k$.
  Consider the map $F:g \mapsto g^{*} g$. We claim that there is a choice of $\omega_1, c_1, c_2$ such that 
  \begin{align}
  \fS_{\omega_{0},C_1,C_2} \subseteq     F(K A^{(1)}_{\omega_{1}} \underline{A}^{\mathbb{R}}_{c_1} N_{c_2} ).
\end{align}

  Let $C',C>0$ be such that $C' \ge \N(h)\ge C$ for all $h \in \omega_0$ (See Remark \ref{re:normOmegaoaway}). Set $$\omega_{1} = \{  a \in D^{(1)}_{\mathbb{R}} \ | \ \tfrac{a^*a}{\Tr(a^{*} a)} \in \omega_0\}.$$ This is a compact set, because $a \in \omega_1$ implies that
  \begin{align}
    \N\left(\frac{a^{*}a}{\Tr(a^{*}a)}\right) = \frac{1}{\Tr(a^{*}a)^{d}} \in \N(\omega_0)   \Rightarrow  (C')^{\frac{1}{d}} \le \Tr(a^{*} a ) \le C^{\frac{1}{d}}.
  \end{align} 
  Now set $c_1 = \sqrt{C_1 (C'/C)^{\frac{1}{d}}}$ and  $c_2 = C_2$. Let $t^*dt \in \fS_{\omega_0, C_1,C_2}$ where $t^{*} d t$ is the Cholesky decomposition, then $t \in N_{c_2}$. Write $d = a' a$ uniquely for $a' \in \underline{A}^{\mathbb{R}}$ and $a \in A^{(1)}$. Then
  \begin{align}
    & \frac{d_{ii}}{\Tr(d_{ii})} = \frac{(a_{ii}^{*}a_{ii}) (a'_{ii})^{2}}{\Tr(a_{ii}^{*}a_{ii})(a'_{ii})^{2}} = \frac{a_{ii}^{*}a_{ii} }{\Tr(a_{ii}^{*}a_{ii})} \in \omega_0 \Rightarrow a_{ii} \in \omega_{1} ,\\
    & \frac{\Tr(d_{ii})}{\Tr(d_{i+1,i+1})}  = \frac{\Tr(a_{ii}^{*}a_{ii}) (a'_{ii})^{2}}{\Tr((a_{i+1,i+1})^{*}a_{i+1,i+1})(a'_{i+1,i+1})^{2}} \le \frac{C^{\frac{1}{d}}}{(C')^{\frac{1}{d}}} c_1^{2}  = C_1,\\
    & \Tr(t_{ij}^{*}t_{ij})   \le c_2 = C_2.
  \end{align}

  Hence, $t^{*} d  t = F( \kappa a a' t)$ for any $\kappa \in K$ and the above choice of $a \in A^{(1)}_{\omega_1}, a'\in \underline{A}^{\mathbb{R}}_{c_1}$ and this settles the claim.

  Now for any $g \in G$, we know from Theorem \ref{th:MS} that for some $b \in M_k(\mathcal{O})$ we have $b^{*}g^{*}gb \in \fS \Rightarrow b^{*}g^{*}gb \in K A^{(1)}_{\omega_1} \underline{A}_{c_1}^{\mathbb{R}} N_{c_2} \Rightarrow gb \in K( K A^{(1)}_{\omega_1} \underline{A}_{c_1}^{\mathbb{R}} N_{c_2} )= K A^{(1)}_{\omega_1} \underline{A}_{c_1}^{\mathbb{R}} N_{c_2}   $.
  From Lemma \ref{le:normonecorrection}, we know that we can find finitely many $b_1, b_2, \dots, b_m \in M_k(D)$ such that $b=b'b_i$ for some $b' \in \Gamma$ and for some $1 \le i \le m$. This implies that $g b' \in \bigcup_{i=1}^{m}  (K A^{(1)}_{\omega_1} \underline{A}_{c_1}^{\mathbb{R}} N_{c_2} )b_i^{-1}$. 
  
  Observe that $N(gb') = 1$, whereas for $(\kappa a a' n )b_{i}^{-1} \in(  K A^{(1)}_{\omega_1} \underline{A}_{c_1}^{\mathbb{R}} N_{c_2}) b^{-1}_{i}$ we have $\N(\kappa a a' n b_{i}^{-1}) = \N(a')/\N(b_{i})$. So $\N(b_i) > 0$ and $(\kappa a a' n )b_{i}^{-1} \in(  K A^{(1)}_{\omega_1} \underline{A}_{c_1}^{\mathbb{R}} N_{c_2}) b^{-1}_{i} \cap G = \N(b_{i})^{\frac{1}{dk}}(  K A^{(1)}_{\omega_1} {A}_{c_1}^{\mathbb{R}} N_{c_2}) b^{-1}_{i}$
  

\end{proof}
\begin{remark}\label{re:liesinQ}
  Note that $\{ b_i\}_{i=1}^{n}$ lie in $GL_k(D) \subseteq SL_k(D_{\mathbb{R}})$. This means that for some $N \in \mathbb{N}$, $Nb_i \in M_k(\mathcal{O})$.
\end{remark}

Now we are in a position to consider $G/\Gamma$ as a probability space.

\begin{proposition}\label{pr:finite_vol}
  The space $G/\Gamma$ carries a unique probability measure that is left-invariant over the action of $G$.
\end{proposition}
\begin{proof}
  The Haar measure of $G$ restricts to left-invariant measure on $G/\Gamma$ since $\Gamma$ is discrete inside $G$. Since $\fS^{*} \subseteq G$ surjects onto $G/\Gamma$, it is sufficient to show that $\fS^{*}$ has a finite measure in $G$.

  The set $\fS^{*}$ is just a union of finitely many translates of $\fS^{1}$. So let us show that $\fS^{1} \subseteq G$ has finite measure. This is to show that the following integral is convergent.
  \begin{align}
    \int_{N_{c_2}} \int_{A^{(1)}_{\omega_{1}}} \int_{A^{\mathbb{R}}_{c_1}} \int_{K}\left( \prod_{i<j}^{}  \frac{ a'_{ii}}{  a'_{jj} }  \right)^{d}  d\kappa da' da dn .
  \end{align}

We can separate the variables in the above integral. Observe that all the integrals other than the one over $A_{c_1}^{\mathbb{R}}$ is over a compact set so must be finite. It simply remains to be shown that the following integral is finite.
\begin{align}
    \int_{A^{\mathbb{R}}_{c_1}} \left( \prod_{i<j}^{}  \frac{ a'_{ii}}{  a'_{jj} }  \right)^{d}  da'.
\end{align}

  In that case, the above integral becomes
\begin{align}
    \int_{A^{\mathbb{R}}_{c_1}} \left( \prod_{i<j}^{}  \frac{ a'_{ii}}{  a'_{jj} }  \right)^{d}  da'.
\end{align}
The group $A^{\mathbb{R}}$ is topologically isomorphic to $(\mathbb{R}^{>0})^{k-1}$, but let us make this identification in the following slightly convoluted manner to make the integral easier for us.
\begin{align}
  A^{\mathbb{R}} \rightarrow&  \  (\mathbb{R}^{>0})^{k-1} \\
  a' \rightarrow  & \ \frac{a'_{ii}}{a'_{(i+1)(i+1)}} 
\end{align}
The above is an isomorphism of locally compact topological groups, and therefore the Haar measure $da'$ can be replaced by a Haar measure of $(\mathbb{R}^{>0})^{k-1}$. Write $y_i = a'_{ii}/a'_{(i+1)(i+1)}$ and now all that remains is to see that the following is a finite integral.
\begin{align}
  \int_{0}^{c_1} \int_{0}^{c_1} \dots \int_{0}^{c_1} \left( \prod_{i<j} y_i^{d} y_{i+1}^{d} \dots y_{j-1}^{d} \right) \frac{dy_1}{y_1} \frac{dy_2}{ y_2} \dots \frac{dy_{k-1}}{y_{k-1}}.
\end{align}
\end{proof}

\begin{remark}
  Proposition \ref{pr:finite_vol} also follows directly \cite{BHC62}, once we know that our group $G$ admits no non-trivial $\mathbb{Q}$-characters.
\end{remark}

We will now prove the following theorem, which is a generalization of the Siegel mean value theorem first presented in $\cite{Sie45}$. 

\begin{theorem}
  \label{th:siegel}
  Let $D$ be a $\mathbb{Q}$-division algebra containing an order $\mathcal{O} \subseteq D$. Let $G = SL_k(D_{\mathbb{R}})$ and $\Gamma = SL_k(\mathcal{O})$, for some $k \ge 2$. Let $dg$ be the probability measure on $G/\Gamma$ that is left-invariant under $G$ action. Then for any $f \in C_{c} (D_{\mathbb{R}}^{k})$, we obtain that 
  \begin{align}
    \int_{G / \Gamma}\left( \sum_{v \in g \mathcal{O}^{k} \setminus \{ 0\}} f(v) \right) dg =  \int_{D_{\mathbb{R}}^{k}}^{} f(x) dx,
  \end{align}
  where $dx$ is a Lebesgue measure on $D_{\mathbb{R}}^{k}$ with respect to which $\mathcal{O}^{k}$ has a covolume of $1$.
\end{theorem}
\begin{remark}\label{re:siegelf}
  Through a small application of the dominated convergence theorem, one can take $f$ to be any Riemann integrable compactly supported function.
\end{remark}

Let us slightly rephrase the theorem. Given a function $f: D_{\mathbb{R}}^{k} \rightarrow \mathbb{R}$ that is compactly supported and continuous, one can make the function $\Phi_{f} : G/\Gamma \rightarrow \mathbb{R}$ given by 
\begin{align}
  \Phi_f(g \Gamma ) = \sum_{v \in g \mathcal{O}^{k} \setminus \{ 0\}}^{} f(v).
\end{align}
This function exists, i.e. does not diverge for any $g \Gamma$, because $f$ is compactly supported and locally it is a finite sum of some evaluations of $f$ and so it is continuous.
The theorem above simply states that $\Phi_f$ has a finite expectation value on $G/\Gamma$
and moreover the expectation is just equal to the integral of $f$. That is,
\begin{align}
  \int_{G/\Gamma}{ \Phi_f(g \Gamma)} dg =  \int_{D_{\mathbb{R}}^{k}}^{} f(x) dx.
\end{align}

Note that if the theorem is indeed true, and if we replace $f$ by an $\varepsilon$-dilate of $f_{\varepsilon}$, i.e. a function $x \mapsto f(\varepsilon x)$ for some $\varepsilon> 0$, we observe that
\begin{align}
  & \int_{D_{\mathbb{R}}^{k}}   f( \varepsilon x) dx   =  \varepsilon^{-dk} \int_{D_{\mathbb{R}}^{k}}^{} f(x) dx  \\ 
   \Rightarrow  &   \int_{G/\Gamma}   \Phi_{f_\varepsilon} d g      = \varepsilon^{-dk} \int_{G/\Gamma}^{} \Phi_f(g\Gamma) d g \\
   \Rightarrow  & \int_{G/\Gamma}^{} \Phi_f(g\Gamma) dg = \int_{G/\Gamma}^{ }  \left(  \varepsilon^{dk} \sum_{y \in \varepsilon g \mathcal{O}^{k} \setminus \{ 0\}}^{} f(  y) \right) dg\Gamma   \label{eq:epsShift}.
\end{align}

Now note that the following limit holds for all $g\Gamma \in G/\Gamma$.
\begin{align}
  \lim_{\varepsilon \rightarrow 0} \left( \varepsilon^{dk} \sum_{y \in \varepsilon g\mathcal{O}^{k} \setminus \{ 0\}}^{} f(y)\right) = \int_{D^k_{\mathbb{R}}}^{} f(x) dx .
\end{align}

Hence, this inspires us to try to use the dominated convergence theorem to prove Theorem \ref{th:siegel}. First, let us try to establish Equality \ref{eq:epsShift} through some other means. The following two lemmas will help us finish the proof of Theorem \ref{th:siegel}.

\begin{lemma}\label{le:dominate}
  Whenever $f \in C_c(D^k_{\mathbb{R}})$, the function $\Phi_f$ is absolutely integrable on $G/\Gamma$. That is, the integral of $| \Phi_f |$ is finite. Furthermore, for any $0 < \varepsilon \le 1$, the function $ \varepsilon^{dk}\Phi_{f_{\varepsilon}}$ is uniformly dominated (independent of $\varepsilon$) by an absolutely integrable function on $G/\Gamma$.
\end{lemma}

\begin{lemma}\label{le:epsind}
  For any $\varepsilon > 0$, and $f, \Phi_f, G/\Gamma$ as before, then we have
  \begin{align}
    \int_{G/\Gamma}^{} \Phi_f(g\Gamma) dg = \int_{G/\Gamma}^{ }  \left(  \varepsilon^{dk} \sum_{y \in \varepsilon g\mathcal{O}^{k} \setminus \{ 0\}}^{} f(  y) \right) dg   = \int_{G/\Gamma}^{} \left(  \varepsilon^{dk} \Phi_{f_{\varepsilon}}(g\Gamma) \right) d g .
     \label{eq:varepsilontozero}
   \end{align}
\end{lemma}

Before proving either of the lemmas, let us show how Theorem \ref{th:siegel} is implied by them.
\begin{proof} (of Theorem \ref{th:siegel})

  Just take $\varepsilon \rightarrow 0 $ in Equation (\ref{eq:varepsilontozero}). By Lemma \ref{le:dominate}, we are guaranteed the following exchange of limits.
  \begin{align}
    \int_{G/\Gamma}^{ } \lim_{\varepsilon \rightarrow 0} \left( \varepsilon^{dk} \Phi_{f_{\varepsilon}}(g\Gamma) \right) dg = \lim_{\varepsilon \rightarrow 0} \int_{G/\Gamma} \varepsilon^{dk} \Phi_{f_{\varepsilon}} ( g\Gamma) d g \stackrel{*}{=} \lim_{\varepsilon \rightarrow 0} \int_{G/\Gamma}^{} \Phi_{f}(g\Gamma) dg .
  \end{align}
  The equality marked with $*$ is due to Lemma \ref{le:epsind}. The final expression on the right is independent of $\varepsilon$ and therefore is equal to the limit. Whereas by the theory of the Riemann integral, we have that for any $g\Gamma \in G/\Gamma$, since $\cov_{dx} ( D^k_{\mathbb{R}} / g\Gamma) = 1$,
  \begin{align}
    \lim_{\varepsilon \rightarrow 0} \varepsilon^{dk} \Phi_{f_{\varepsilon}}(g\Gamma) =  \lim_{\varepsilon \rightarrow 0} \varepsilon^{dk} \sum_{v \in  \varepsilon g\mathcal{O}^{k} \setminus \{ 0\} }^{} f(v) = \int_{D^k_{\mathbb{R}}}^{} f(x ) dx.
  \end{align}
  Hence, the pointwise limit of $\varepsilon^{dk} \Phi_{f_{\varepsilon}}(g\Gamma) $ is the constant value $\int_{D^k_{\mathbb{R}}}^{}f(x) dx$. Putting this together gives us the required result.
\end{proof}

We will now give detailed proofs of the two given lemmas.

\begin{proof} (of Lemma \ref{le:dominate})

  We will directly prove that $\varepsilon^{dk} \Phi_{f_{\varepsilon}}$ is dominated, under the assumption that $f$ is non-negative everywhere on $D_{\mathbb{R}}^{k}$. We will integrate on $G/\Gamma$ with respect to the Haar measure introduced shortly before. Note that, this measure may not be a probability measure on $G/\Gamma$, but the difference is only that of correction by a constant.

  Recall $K, A^{\mathbb{R}},A^{(1)}, N$ as discussed in Proposition \ref{pr:finite_vol}. Let $M = \sup_{v \in D_{\mathbb{R}}^{k}} |  f(v)|$ and $R > 0$ be $f$ is supported inside $B_{R}(0) \subset D_{\mathbb{R}}^{k}$ (open ball of radius $R$ around 0 with respect to the trace norm).

  Then we get that for some constant $C$, which arises out of the choice of scaling\footnote{In fact, $C$ would be equal to the reciprocal of the volume of $G/\Gamma$ in terms of the measure from Corollary \ref{co:haar}} of the Haar measure on $G/\Gamma$
  \begin{align}
    \int_{G/\Gamma}^{} \varepsilon^{dk} \Phi_{f_{\varepsilon}}(g \Gamma ) dg \le  & \varepsilon^{dk} \int_{\fS^{*}}^{} \left( \sum_{v \in g \mathcal{O}^{k} \setminus \{ 0\}}^{} f_{\varepsilon}(v) \right) dg \\
    \le M  & \varepsilon^{dk} \int_{\fS^{*}}^{} \left(  (g \mathcal{O}^{k}\setminus \{ 0\}) \cap B_{R/ \varepsilon }(0)\right) dg \\
    \le &   M \varepsilon^{dk} \sum_{i=1}^{n}\int_{\fS^{1}}  \#  \left(  (g b_{i}^{-1}\mathcal{O}^{k}) \cap B_{R/ \varepsilon }(0)\right)dg\\
    =    C M \sum_{i=1}^{m}\varepsilon^{dk} \int_{N_{c_2}} & \int_{A_{\omega_1}^{(1)}} \int_{A^{\mathbb{R}}_{c_1}} \int_{K}  \# \left( \N(b_{i})^{\frac{1}{dk}}(\kappa a' a n) b_i^{-1}\mathcal{O}^{k} \cap B_{R/\varepsilon}(0)\right)\ \prod_{i < j}^{} \left( \frac{a'_{ii}}{a'_{jj}} \right)^{d} d\kappa da' da dn .
  \end{align}
  Since $B_{R/\varepsilon}(0)$ is invariant under $K$ ( $\Tr(x^{*} x) = \Tr(x^{*} (\kappa^{*}\kappa) x)$ for $\kappa \in K$, $x \in M_k(D_{\mathbb{R}})$), we know that for any $\kappa \in K$,
  \begin{align}
    \# \left( \N(b_i)^{\frac{1}{dk}} (\kappa a' a n) b_i^{-1}\mathcal{O}^{k} \cap B_{R/\varepsilon}(0) \right) = \# \left( \N(b_i)^{\frac{1}{dk}}(a'an) b^{-1}_i\mathcal{O}^{k} \cap B_{R/\varepsilon}(0) \right).
  \end{align}

  Because of Remark \ref{re:liesinQ}, we know that there exists some $N \in \mathbb{N}$ such that $b_i^{-1} \mathcal{O}^{k} \subseteq \frac{1}{N} \mathcal{O}^{k}$ for every $1 \le i \le n$. Hence, this tells us that 
  \begin{align}
    \# \left( \N(b_{i})^{\frac{1}{dk}}(a'an) b^{-1}_i\mathcal{O}^{k} \cap B_{R/\varepsilon}(0) \right) \le & \#\left(  (a'an) \frac{\N(b_i)^{\frac{1}{dk}}}{N} \mathcal{O}^{k} \cap B_{R/ \varepsilon  }(0) \right)  \\
    = &  \# \left( (a'an) \mathcal{O}^{k} \cap B_{R_{i}/\varepsilon}(0) \right), \text{ where } R_i = \frac{R N}{ \N(b_{i})^\frac{1}{dk}}.
  \end{align}

  Now consider the set $Y= \{ a'an(a')^{-1} \ | \ {a' \in A^{\mathbb{R}}_{c_1}, a \in A^{(1)}_{\omega_1}, n \in N_{c_2}} \} \subseteq G$. For $y \in Y$, note that $y_{ij} = a'_{ii}a_{ii} n_{ij} (a'_{jj})^{-1} $. 
  Hence, $\Tr(y_{ij}^{*} y_{ij}) = \left(\frac{a'_{ii}}{a'_{jj}}\right)^{2}\Tr\left( (a_{ii} n_{ij})^{*} (a_{ii}n_{ij})\right)$.
  Here, $\left( \frac{a'_{ii}}{a'_{jj}} \right)$ is a positive number bounded by $c_1^{j-i}$because of the construction of $A^{\mathbb{R}}_{c_1}$, and the other term is bounded because it continuously depends on $a_{ii} n_{ij}$ which lie in a compact set. Hence, overall the set $Y$ must lie inside a relatively compact set of $G$. Furthermore, the set $Y$ is only dependent on $c_1, c_2$ and $\omega_1$, which are only dependent on $D,\mathcal{O}$ and $k$.

  Now what do we want to do with this set $Y \subseteq G$? So let $R' > 0$ be a radius such that for each index $i$,
  $Y^{-1} B_{R_i}(0) \subseteq B_{R'}(0) \Rightarrow Y^{-1}B_{R_i/\varepsilon}(0) \subseteq B_{R'/\varepsilon}(0)$. Then, we write that 
  \begin{align}
    \# \left( (a'an) \mathcal{O}^{k} \cap B_{R_i/\varepsilon}(0) \right) = & \# \left( (a'an(a')^{-1})a' \mathcal{O}^{k} \cap B_{R_i/\varepsilon}(0) \right) \\
    \le & \# \left( a' \mathcal{O}^{k} \cap Y^{-1} B_{R_i/\varepsilon}(0) \right)  \\
    \le & \# \left( a' \mathcal{O}^{k} \cap B_{R'/\varepsilon}(0) \right).
  \end{align}

 The value of this last expression is equal to the number of integer solutions $(x_1, \dots, x_k) \in \mathcal{O}^{k}$ such that 
  \begin{align}
    \sum_{i=1}^{k}   {a'_{ii}}^{2} \Tr_{D_{\mathbb{R}}}(x_i^{*} x_i) \le \frac{ R'^{2}}{ \varepsilon^{2}}.
  \end{align}

  This is the number of points in a lattice intersecting with some ellipsoid. By considering a bounding cuboid of the ellipsoid, an upper bound for the number of solutions is the following product.
  \begin{align}
    \prod_{i=1}^{k}\#\left\{  x \in \mathcal{O} \ | \ \Tr_{D_{\mathbb{R}}}(x^{*} x) \le \frac{ {R'}^{2} }{ {a'_{ii}}^{2} \varepsilon^{2} }  \right\}.
  \end{align}

  Each term in the product is the number of points in a ball of radius $R'/a_{ii}'\varepsilon$ in a $d$-dimensional $\mathbb{R}$-vector space. Hence, there exist constants $B_1, B_2 > 0$ depending only on $\mathcal{O}, D $ such that 
  \begin{align}
    \#\left\{  x \in \mathcal{O} \ | \ \Tr_{D_{\mathbb{R}}}(x^{*} x) \le \frac{ {R'}^{2} }{ {a'_{ii}}^{2} \varepsilon^{2} }  \right\} \le B_1 + B_2 \left( \frac{R'}{a'_{ii} \varepsilon} \right)^{d} ,
  \end{align}

  and therefore 
  \begin{align}
    & \int_{G/\Gamma} \varepsilon^{dk} \Phi_{f_{\varepsilon}}(g \Gamma) dg  \\ 
    & \le \sum_{i=1}^{m} CM\varepsilon^{dk} \int_{N_{c_{2}}} \int_{A^{(1)}_{\omega_1}} \int_{A^{\mathbb{R}}_{c_1}} \int_{K} \left( \prod_{i=1}^{k} \left( B_1 + B_2\left( \frac{R'}{a'_{ii} \varepsilon} \right)^{d}   \right) \right) \prod_{i < j}^{} \left( \frac{a'_{ii}}{ a'_{jj}} \right)^{d} d\kappa da' da dn \\ 
    & =m CM\int_{N_{c_{2}}} \int_{A^{(1)}_{\omega_1}} \int_{A^{\mathbb{R}}_{c_1}} \int_{K} \left( \prod_{i=1}^{k} \left( B_1 \varepsilon^{d} + B_2\left( \frac{R'}{a'_{ii} } \right)^{d}   \right) \right) \prod_{i < j}^{} \left( \frac{a'_{ii}}{ a'_{jj}} \right)^{d} d\kappa da' da dn
  \end{align}

  Now $\varepsilon \le 1 \Rightarrow  B_{1} \varepsilon^{d} \le B_1$. Therefore, we can bound the integral above by 
  \begin{align}
    & \int_{G/\Gamma} \varepsilon^{dk} \Phi_{f_{\varepsilon}}(g \Gamma) dg  \\ 
    & \le CM\int_{N_{c_{2}}} \int_{A^{(1)}_{\omega_1}} \int_{A^{\mathbb{R}}_{c_1}} \int_{K} \left( \prod_{i=1}^{k} \left( B_1  + B_2\left( \frac{R'}{a'_{ii} } \right)^{d}   \right) \right) \prod_{i < j}^{} \left( \frac{a'_{ii}}{ a'_{jj}} \right)^{d} d\kappa da' da dn
  \end{align}

  This last integral does not contain any appearance of $\varepsilon$. Note that for a decoposition of $g = \kappa a' a n$, the matrix $a'$ is unique. Therefore, some appropriate scaling of the function $g \mapsto  \prod_{i=1}^{k}\left(B_1 + B_2(R' {a'_{ii}}^{-1})^{d}\right)$ on a fundamental domain of $G/\Gamma$ is a dominating function of $\varepsilon^{dk} \Phi_{f_{\varepsilon}}$, if we prove that the integral above is convergent.

  The sets $K, A^{(1)}_{\omega_1}$ and $N_{c_2}$ are compact and hence $\int_{K}{ dk} \int_{N_{c_2}}^{}dn $ and $\int_{A^{(1)}_{\omega_{1}}}da$ are finite. Hence, we just need to show the finiteness of
  \begin{align}\label{eq:integral}
  \int_{A^{\mathbb{R}}_{c_1}}  \left( \prod_{i=1}^{k} \left( B_1  + B_2\left( \frac{R'}{a'_{ii} } \right)^{d}   \right) \right) \prod_{i < j}^{} \left( \frac{a'_{ii}}{ a'_{jj}} \right)^{d}  da' .
\end{align}


Let us first do this for the case $k=2$, that is when $G=SL_2(D_{\mathbb{R}})$ and $\Gamma = SL_2(\mathcal{O})$. In that case, $A^{\mathbb{R}} \simeq \mathbb{R}^{>0}$, and we can parametrize it as $a'_{11}={ a'_{22}}^{-1} = t$. The condition $a_{11}'\le c_1 a'_{22}$ is just saying that $t^{2} \le c_1$. The measure $da'$ is $\frac{1}{t}dt$. So the integral becomes
\begin{align}
  \int_{0}^{\sqrt{c_1}}  \left( B_1  + B_2\left( \frac{R'}{t } \right)^{d}\right) \left( B_1 + B_2(R't)^{d} \right)  t^{2d} \frac{dt}{t}  
\end{align}
which is clearly finite.

For the general $k$, here it goes.
We will use the coordinates of integration from Proposition \ref{pr:finite_vol}. 
Define $y_i =  { a'_{ii}}/{ a'_{i+1,i+1}}$ for $i \in \{ 1,\dots, k-1\}$. Then we have 
\begin{align}
  \begin{bmatrix} 1 & -1 &   &   &   &   \\
    & 1 & -1 &   &   &   \\  
  & & 1 & -1 & & \\
  & &  & \ddots & & \\
& & & & 1  & -1  \\
 1 & 1 & 1 &1  &1  &2  \end{bmatrix}
\begin{bmatrix} \log a'_{11}  \\ \log a'_{22} \\ \log a'_{33} \\   \vdots \\  \\ \log a'_{k-1,k-1} \end{bmatrix} =  \begin{bmatrix} \log y_{1} \\ \log y_{2} \\ \log y_{3} \\ \vdots \\ \\ \log y_{k-1} \end{bmatrix}.
\end{align}

The last row is so because $\sum_{i =1 }^{k-1}\log a'_{ii} = 0$. The inverse of the square matrix above is 
\begin{align}\frac{1}{k}
  \begin{bmatrix} k-1 & k-2& k-3& k-4& & 1 \\
-1  & k-2  & k -3 & k-4 & \dots & 1 \\
 -1 & -2 &k-3 &k-4 &  & 1 \\
 -1 & - 2& -3 &k-4 & & \\
 &\vdots & & & & \vdots \\
 -1& -2  &  -3 & - 4&\dots & 1 \end{bmatrix}  = \left[  \mathbb{I}_{i \le j} -  \frac{j}{k}\right]_{i,j=1}^{k-1} ,
\end{align}
and the determinant is $k$. 

Then we get from the above calculations that 
\begin{align}
  \frac{a'_{ii}}{a'_{jj}} = \prod_{r = i}^{ j-1} y_{r} \Rightarrow \prod_{i < j}^{} \frac{a'_{ii}}{a'_{jj}} = \prod_{j = 1}^{k-1} y_{j}^{j (k-j)} .
\end{align}
From the matrix inverse, we get that for $i \in \{ 1,2,\dots,k-1\}$
\begin{align}
  a'_{ii} = e^{\log a'_{ii}}  = e^{\sum_{j = i}^{ k-1} \log y_j  - \sum_{j = 1}^{ k-1}  \frac{j}{k}\log y_j} =\prod_{ j =i   }^{k-1} y_{j} \prod_{j=1}^{k-1} y_{j}^{-\frac{j}{k}}.
\end{align}
and 
\begin{align}
  a'_{kk} = e^{ - \sum_{j=1}^{k-1} \frac{j}{k}\log y_j }  = \prod_{j=i}^{k-1} y_{j}^{-\frac{j}{k}}.
\end{align}
Finally, the Haar measure $da'$ can be taken to be $\prod_{i=1}^{k-1}\frac{dy_{i}}{y_{i}}$.

Putting this all together, the integral (\ref{eq:integral}) becomes
\begin{align}
  \int_{0 < y_i \le c_1 }^{} \left[ \prod_{i =1  }^{ k -1 } \left( B_1 + B_2{R'}^{d}  \prod_{j=1}^{k-1} y_{j}^{\frac{jd}{k}} \prod_{j=i}^{k-1} y_{j}^{-d} \right) \right] \left( B_1 + B_2 {R'}^{d}\prod_{j=i}^{k-1}y_j^{\frac{jd}{k}} \right)\left(\prod_{j=1}^{k-1} y_{j}^{jd(k-j)}  \right)\prod_{i =1 }^{ k-1} \frac{dy_{i}}{y_{i}}.
\end{align}

Then the finiteness of the integral can be shown by simply chasing the powers of each $y_{i}$ and showing that it is greater than $0$.

Distributing the first product over subsets $I \subseteq \{ 1,2,\dots,k-1\}$ gives us
\begin{align}
  = &  \sum_{I \subseteq \{ 1,2,\dots,k-1\}}^{} 
  \int_{ 0 < y_i \le c_1}\ B_1 ^{k - 1 } 
  \left( \prod_{i \in I}^{} \frac{B_2 {R'}^{d} {\prod_{j=1}^{k-1} y_{j}^\frac{jd}{ k} }}{B_1 {\prod_{j=i}^{k-1} y_j^{d}} } \right) 
  \left(  B_{2} {R'}^{d} {\prod_{j =1 }^{ k- 1}y_j^{\frac{jd}{k}}} + B_1 \right)
  \left(\prod_{j=1}^{k-1} y_{j}^{jd(k-j)}  \right) 
  \prod_{ i =1 }^{ k-1}\frac{d y_i}{y_i} \\ 
  = &  \sum_{I \subseteq \{ 1,2,\dots,k-1\}}^{} 
  \int_{ 0 < y_i \le c_1}\ B_1 ^{k - 1 } 
  \frac{ \left( B_2 {R'}^{d} \prod_{j=1}^{k-1} y_{j}^\frac{jd}{ k} \right)^{\# I }}{ B_1^{\# I } {\prod_{j=1}^{k-1} y_j^{d \left( \# I_{\le j} \right)}} }  
  \left(  B_{2} {R'}^{d} {\prod_{j =1 }^{ k- 1}y_j^{\frac{jd}{k}}} + B_1 \right)
  \left(\prod_{j=1}^{k-1} y_{j}^{jd(k-j)}  \right) 
  \prod_{ i =1 }^{ k-1}\frac{d y_i}{y_i} .
\end{align}
where we have $I_{\le j} = \{  i \in I \ | \ i \le j \} $. Now in the above expression, for each $I \subset \{ 1,2,\dots, k-1\}$ we have an integration of a sum of two products of some powers of $y_{j}$ and some constant. If we prove that the power of $y_{j}$ in each of those terms is $\ge 0$, then we are done. Note that the power of a $y_j$ for $j \in \{ 1,2,\dots, k-1\}$ in the two summands would be 
\begin{align}
\frac{jd}{k}(\# I ) - d (\# I_{\le j})  + \frac{jd}{k} + jd(k-j) - 1,
\end{align}
and
\begin{align}
  \frac{jd}{k} (\# I) - d (\# I_{\le j}) + jd(k-j) - 1.
\end{align}

It is sufficient to show that the latter is $\ge 0$ for each $I \subseteq \{ 1,2,\dots,k-1\}$ and for each $j$. Rewriting that last expression as 
\begin{align}
  d \left( \frac{j}{k} (\# I) -  (\# I_{\le j}) + j(k-j)\right) - 1.
\end{align}

 

Hence the finiteness of the integral now clearly follows from proving that 
\begin{align}
  (\# I)\frac{j}{k} + j (k-j) - (\# I_{\le j}) \gneq  0 \Leftrightarrow (k-j) \gneq \frac{\# I_{\le j}}{ j}   - \frac{\# I}{k}.
\end{align}
The inequality is indeed true. Combinatorially $\# I_{\le j} \le j$ and $\# I \le k -1 \lneq k $ so the difference $ \left| \frac{\# I_{\le j}}{ j}   - \frac{\# I}{k}\right| \lneq 1$. On the other hand since $j \in \{ 1,2,\dots, k-1\}$, we must have $k-j \ge 1$.

\end{proof}

\begin{proof} (of Lemma \ref{le:epsind}).
  The strategy here is try to exchange the summation over lattice points with the integral over $G$ in our expression $\prod_{G/\Gamma}^{}(\sum_{v \in g\mathcal{O}^{k} \setminus \{ 0\}}^{} f(v ) )dg$. This will obtain Theorem \ref{th:siegel} up to a constant.

  Note that, $\Gamma$ is precisely the set of linear transformations in $G$ that preserve the lattice $\mathcal{O}^{k} \subseteq D_{\mathbb{R}}^{k}$. Now consider the orbit set 
  \begin{align}
    \Gamma \backslash \mathcal{O}^{k}  = \{ [v] := \Gamma v \ | \ v \in \mathcal{O}^{k}\}.
  \end{align}
For any $g \in G$ we also have that $ g\Gamma g^{-1}$ is the group of symmetries in $G$ that preserve $g \mathcal{O}^{k}$ and 
$$  g \Gamma g^{-1} \backslash g \mathcal{O}^{k}   = g \left( \Gamma \backslash \mathcal{O}^{k}  \right)=\{ g [v] \ | \  [v] \in \Gamma \backslash \mathcal{O}^{k}\}.$$

  With all this, we have that
  \begin{align}
    \int_{G/\Gamma}\left( \sum_{v \in \mathcal{O}^{k}\setminus \{ 0\}}^{} f(v)\right) dg   & = \int_{G/ \Gamma } \left( \sum_{ [v] \in ({\Gamma }\backslash {\mathcal{O}^{k}}) \setminus \{ [0]\} } \ \  \sum_{ y\in [v]}^{} f(g y )  \right) d g\\
    & = \sum_{[v] \in \left( { \Gamma}\backslash {\mathcal{O}^{k}} \right) \setminus \{[0]\} } \int_{G/ \Gamma }^{} \sum_{y \in [v]}^{} f(gy) dg.
  \end{align}

  The above interchange of the sum with integral can be justified via the dominated convergence theorem, since the partial sums are dominated by $ \Phi_{|f|}$ which is integrable by the previous lemma.
  Now since $[v]= \Gamma v \simeq \Gamma / \Gamma_{v}$, where $\Gamma_{v}$ is the stabilizer subgroup of $v \in [v]$ in $\Gamma$, we can write that 
  \begin{align}
    \sum_{[v] \in\left( { \Gamma}\backslash{\mathcal{O}^{k}} \right) \setminus \{[0]\} } \int_{G / \Gamma }\sum_{y \in [v]}^{} f(gy) dg & = \sum_{[v] \in \left( { \Gamma}\backslash{ \mathcal{O}^{k} }\right) \setminus \{ [0]\}} \int_{ G/ \Gamma } \sum_{  h\Gamma_{v}\in \Gamma/\Gamma_{v}}f(g h v) dg \\
    & = \sum_{[v] \in \left( {\Gamma}\backslash{ \mathcal{O}^{k}} \right) \setminus \{ [0]\}} \int_{ G/ \Gamma_{v} }f(gv) dg 
    .\label{eq:whatmeasure}
  \end{align}

  The second equality above merits some explanation. $\Gamma_{v}$ is a discrete subgroup of $G$ and is also unimodular. Hence, there is a unique scaling of a $G$-invariant measure on $G/\Gamma_{v}$ that agrees with the measure given as $$F \mapsto \int_{G/\Gamma}^{} \left( \sum_{[h] \in \Gamma /\Gamma_{v}}^{} F(gh)  \right)dg , \mbox{ for }F\in C_{c}(G/\Gamma_v) .$$
  The measure $dg$ in (\ref{eq:whatmeasure}) refers to this measure which ``unfolds the integral''.

  Now with $G_{v}$ being the stabilizer subgroup of $v$ in $G$ there is a homeomorphism\footnote{This is because $D_{\mathbb{R}}^{k}$ is a locally compact space and the orbit $Gv$ is also locally compact.} $G/G_{v} \simeq D_{\mathbb{R}}^{k}\setminus \{ 0\}$ given by $g G_v \mapsto  gv $. Note that, this works out because $G$ acts transitively on $D^{k}_{\mathbb{R}} \setminus \{ 0\}$. Furthermore, the Lebesgue measure on $D_{\mathbb{R}}^{k} \setminus \{ 0\}$ induces a $G$-invariant measure on $G/G_{v}$ implying that $G_v$ is also unimodular. In particular, $G_{v}/\Gamma_{v}$ carries a unique (up to scaling) $G_{v}$-invariant measure, since $\Gamma_v$ being a discrete group must also be unimodular. With all this in place, we can unwind the integral in (\ref{eq:whatmeasure}) again and write
\begin{align}
\int_{G/\Gamma}{\Phi_{f}(g \Gamma) dg} = \sum_{[v] \in \left( \Gamma \backslash \mathcal{O}^{k} \right) \setminus \{ [0]\}}^{} \int_{G/G_{v}}^{} \int_{G_v / \Gamma_{v}} f(h_2 h_1 v ) dh_1 dh_2,
\end{align}
where $dh_1$ and $dh_2$ are scaled appropriately for the equality to make sense. But $h_1 v = v$, since $h_1 \in G_v$ and hence
\begin{align}
  \int_{G/\Gamma}^{ } \Phi_{f}(g \Gamma)d \Lambda = \sum_{[v]\in \left( {\Gamma}\backslash{ \mathcal{O}^{k} } \right) \setminus  \{ 0\}} \left( \int_{G / G_{v}} f(h_2 v) dh_2\right) \left( \int_{G_{v} / \Gamma_{v}} dh_1  \right).
\end{align}

Now since $G/G_v \simeq V \setminus \{ 0\}$, the integral $\int_{G/G_{v}}^{} f(h_1 v) dh_1 = c_v \int_{D^{k}_{\mathbb{R}} \setminus \{ 0\}}^{} f(x) dx $ for some $c_v > 0$. On the other hand, because we know that this integral is absolutely convergent by the previous lemma, we must have another positive constant $ 0 < c'_v = \int_{G_{v}/\Gamma_{v}}^{} dh_2 = \vol_{dh_2}(G_{v}/\Gamma_{v})  < \infty$ (since if the volume is infinite, the integral on the left wouldn't be finite). Hence, we finally obtain that 

\begin{align}
  \int_{G/\Gamma}^{ } \Phi_{f}(g \Gamma)d g & = \sum_{[v]\in \left( \Gamma \backslash \mathcal{O}^{k} \right) \setminus \{ [0]\}} c_v c'_{v} \int_{ D^{k}_{\mathbb{R}} \setminus \{ 0\}} f(x ) dv \\
& =  \int_{D^{k}_{\mathbb{R}}}^{ } f(x) dx  \left( \sum_{[v] \in \left( {\Gamma}\backslash { \mathcal{O}^{k} }\right) \setminus \{ [0]\}} c_v c'_v \right)\\
& = C \int_{D^{k}_{\mathbb{R}}}^{} f(x) dx .
\end{align}
Here $C$ is a constant which must be finite since the integral is. In fact, $C = 1$, but that's not important for the conclusion of the lemma.  Since $\int_{D^{k}_{\mathbb{R}}}^{} f_{\varepsilon}(x) dx = \varepsilon^{-dk } \int_{D_{\mathbb{R}}^{k}}^{}f(x)dx$, the result follows from simple rearranging.
\end{proof}
\section{Lower bounds on lattice packing efficiency}
\subsection{Overall strategy}
\label{se:strategy}

The main idea that we will employ is in the form of the following proposition.

\begin{proposition}\label{pr:bound}
  Let $\mathcal{O} \subseteq D$ be an order in a division algebra and let $G_0 \subseteq D$ be a finite multiplicative subgroup of $\mathcal{O}$. Then for any $\varepsilon > 0$, there exists a lattice packing in dimensions $d= 2\dim_{\mathbb{Q}} D$ whose packing efficiency is at least $\frac{1}{2^{d}} ( \#G_{0} )-\varepsilon $.
\end{proposition}
\begin{proof}
  What we will show is that there exists a positive definite quadratic form on $D_{\mathbb{R}}^{2}$ and a unit covolume lattice $\Lambda_0$ (with respect to this quadratic form), such that for some ball $B_{R}(0)$ in this quadratic form having a volume $\# G_{0}- \varepsilon$, the lattice and the ball intersect only at $\{ 0\}$. If we prove this, then we get that the balls $B_{R/2}(v_1),B_{R/2}(v_2)$ are disjoint for any distinct $v_1,v_2 \in \Lambda_0$ and hence $\bigsqcup_{v \in \Lambda_0} B_{R/2}(v)$ forms a lattice packing whose packing efficiency will be 
  \begin{align}
    \frac{\vol B_{R/2}(0)}{\vol(D^{2}_{\mathbb{R}} / \Lambda_{0})} = 2^{-d} (\# G_{0} - \varepsilon ).
  \end{align}

  Consider the left-action of $G_{0}$ on $D_{\mathbb{R}}^{2}$ via $g.(v_1,v_2) = (v_1g^{-1},v_2g^{-1})$. This action is $\mathbb{R}$-linear and therefore it is possible to start with any positive-definite quadratic form on $D_{\mathbb{R}}^{2}$ and average over $G_{0}$ and make it $G_{0}$-invariant. After appropriate scaling, the lattice $\mathcal{O}^{2}$ will have a unit covolume with respect to the measure induced by this form. We fix this as the form on $D^{2}_{\mathbb{R}}$ as mentioned above.

  Now let $B_{R}(0)$ be the ball of volume $\# G_{0} - \varepsilon$ and let $f$ be the indicator function of $B_{R}(0)$. Then, we get from Theorem \ref{th:siegel} and Remark \ref{re:siegelf}
  \begin{align}
    \int_{G / \Gamma}\left( \sum_{v \in g \mathcal{O}^{2} \setminus \{ 0\}} f(v) \right) dg =  \int_{D_{\mathbb{R}}^{2}}^{} f(x) dx = \# G_{0} - \varepsilon.
  \end{align}

  However, note that for any $g \in G$, the lattice $g \mathcal{O}^{2}$ is $G_{0}$-invariant under the left-action defined above. Furthermore, the $G_{0}$-orbit of any non-zero element of $\mathcal{O}^{2}$ is of size $\# G_{0}$ because $\mathcal{O}^{2}$ and $G_{0}$ are made of elements of the division algebra $D$. Therefore, $\sum_{v \in g \mathcal{O}^{2} \setminus \{ 0\}}^{} f(v)$ lies in $\{ \#G_{0}, 2( \# G_{0}), 3(\# G_{0}), \dots \}$. Since the average is strictly less than $\# G_{0}$, we get that for some $g_0 \in G$, $\Lambda_0 = g_0 \mathcal{O}^{2} \cap B_{R}(0) = \{ 0\}$ and this is the required lattice.
\end{proof}

That $\varepsilon$ in the above lower bound can be gotten rid of by using Mahler's compactness theorem. 
\begin{theorem}\label{pr:boundbetter}
  Let $(G_{0},\mathcal{O}, D ) $ be as in Proposition \ref{pr:bound}. Then there exists a lattice packing in dimensions $d= 2\dim_{\mathbb{Q}} D$ whose packing efficiency is at least $\frac{1}{2^{d}} ( \#G_{0} )$.
\end{theorem}
\begin{proof}
  Let $\Lambda_n = g_n \mathcal{O}^{2}$ be a unit covolume lattice in $D_{\mathbb{R}}^{2}$ whose packing efficiency is better than $\frac{1}{2^{d}}(\# G_{0}) - \frac{1}{n}$. Since all $\Lambda_n$ are unit covolume and whose packing efficiency is bounded below, we get from Mahler's compactness that up replacing $g_n$ with $g_n \gamma_n$ for some $\gamma_{n} \in \Gamma$, we can force $\{ g_n\}_{n \ge 1}$ to be a relatively compact set in $G$ and therefore it contains a convergent subsequence converging to some point $g \in G$. Since packing efficiency is a continuous function on $G/\Gamma$, we get that $g \mathcal{O}^{2}$ is the required lattice.
\end{proof}

Hence, this gives us a methodology of procuring lower bounds for lattice packings. Any tuple $(G_{0},\mathcal{O},D)$ gives us a packing from Proposition \ref{pr:bound} gives us a valid lower bound for the sphere packing problem in dimension $d= 2 \dim_{\mathbb{Q}} D$, i.e. $c_{d} \ge |G_0|$.

\begin{example}
  For $n\ge 3$, put $ D = \mathbb{Q}(\mu_{n})$, and $\mathcal{O} \subset D $ as its ring of integers, and $G_{0} = \langle \mu_{n}\rangle \simeq {\mathbb{Z}}/{n \mathbb{Z}}$. Hence, in dimension $d = 2 \varphi(n) $, there is a lattice packing of packing efficiency at least $2^{-d}\# G_{0} =2^{-d} n$. This gives us the lower bound in \cite{AV}.
\end{example}

Note that the following ``tightening'' can be done once we have a tuple $(G_{0},\mathcal{O},D)$. When $D$ is a $\mathbb{Q}$-division algebra, the $\mathbb{Q}$-span of $G_{0}$ in $D$ is also a division algebra. Indeed, denote $\mathbb{Q}\langle G_{0}\rangle  \subseteq D$ as the span\footnote{Caution: This is not the group algebra of $G_{0}$. The group algebra of $G_{0}$ over $\mathbb{Q}$ will almost never be a division algebra. More precisely, this is the image of the group algebra under $\mathbb{Q}[G_{0}] \rightarrow D$ induced from the inclusion $G_{0} \hookrightarrow D$}  of $G_{0}$, then any $\gamma \in \mathbb{Q}\langle G_{0}\rangle $ is an invertible $\mathbb{Q}$-map therefore it will map $\mathbb{Q}\langle G_{0}\rangle $ to itself under left-multiplication and therefore must map something to $1_{D}$. Let $\mathbb{Z}\langle G_{0}\rangle  \subseteq \mathcal{O}$ be the $\mathbb{Z}$-span of $G_{0}$, then we get that $(G_{0},\mathbb{Z}\langle G_{0}\rangle ,\mathbb{Q}\langle G_{0}\rangle )$ is another tuple that fits in Proposition \ref{pr:bound}.

Clearly, $\dim_{\mathbb{Q}}\mathbb{Q} \langle G_{0} \rangle \le \dim_{\mathbb{Q}} D $. Therefore, we can get a packing in smaller dimension without losing the packing efficiency. Hence, to get tighter packings it is sufficient to consider the case where the $\mathbb{Q}$-span of $G_{0}$ is precisely $D$. $\mathcal{O}$ can then be taken to be the $\mathbb{Z}$-span of $G_{0}$.

This tightening also shows why it was optimal to consider cyclotomic fields in \cite{AV}. If the division algebra is a general number field $K$ and $G_0 \subseteq K^{*}$ is the group of torsional units, then by Dirichlet's unit theorem we have that $G_0 = \langle \mu_{m}\rangle$ for some $m \in \mathbb{Z}_{> 1}$ and $\mathbb{Q}\langle G_0\rangle = \mathbb{Q}(\mu_{m})$ would be a cyclotomic field.

\label{se:overall}

  \subsection{Cyclic division algebras}\label{ss:CDA}

  This section is going to be a review of cyclic division algebras.

  We know that the Frobenius theorem allows only three finite dimensional $\mathbb{R}$-division algebras, namely $\mathbb{R}, \mathbb{C}$ and $\mathbb{H}$. The only non-trivial and non-commutative extension of $\mathbb{R}$ is $\mathbb{H}$. However, over $\mathbb{Q}$, the story is completely different. There are infinitely many finite dimensional $\mathbb{Q}$-division algebras apart from the finite field extensions of $\mathbb{Q}$. All of these division algebras have the form of a cyclic division algebra. 
  For a thorough introduction, one can refer to \cite{NJ2009}, for instance. 

  We define a cyclic $\mathbb{Q}$-division algebra as the quadruplet $D = (E,F,\sigma, \gamma)$, where 
  \begin{enumerate}
    \item $F$ is a number field over $\mathbb{Q}$, 
    \item $E/F$ is a cyclic extension of degree $n$, i.e. the field extension $E/F$ is Galois and the Galois group is cyclic,
    \item $\sigma$ is a generator of the cyclic group $\Gal(E/F)$ and 
    \item $\gamma \in F^{*}$, with the property that the multiplicative order of $\gamma$ in the group $K^{*}/ N_{F}^{E}(E^{*})$ is exactly $n$. That is, $\gamma^{k} \notin N_{F}^{E}(E^{*})$ for any $k \in \{ 1,2,\dots,n-1\}$ and $\gamma^{n} = N_{F}^{E}(x)$ for some $x\in E^{*}$\label{it:nonnorm}. When this happens we say that $\gamma \in F^{*}$ is a non-norm element. Note that $\gamma^{n} = \N^{E}_{F}(\gamma)$.
  \end{enumerate}

  Consider a formal element $b$ that does not commute with $E$ and satisfies $b^{n} = \gamma$. $D$ is now defined as per the isomorphism
  \begin{align} \label{eq:iDentify} D \simeq E \oplus E b \oplus E b^{2} \oplus \dots \oplus E b^{n-1},\end{align}
  with the rule that 
  \begin{align}
    bl =  \sigma(l)b \text{ for all }l \in E.\label{eq:ruleofD}
  \end{align}
  If we identify $D \simeq E^{n}$ according to the identification (\ref{eq:iDentify}), then for $g = (g_0, g_1, \dots, g_{n-1})$ we observe that for some $x_0 \in E$ we get the following from repeatedly using Equation \ref{eq:ruleofD}.
  \begin{align}
    g b  = &  (g_0 + g_1 b + g_2 b^{2} + \dots + g_{n-1} b^{n-1}) x_0  \\ \\
    = & \begin{bmatrix}  x_0 &   &    &    &   & \  &   &  \\
      & \sigma(x_0) &   &   &   &   &       &    \\
      &  & \sigma^{2}(x_0)  &   &    &   &      &   \\
      &   &  & \sigma^{3}(x_0)  &   &   &   &    \\
      &   &   &  & \sigma^{4}(x_0) &   &   &    \\
    \  &   &   &   &   & \ddots  &  &    \\
    &   &   &   &   &  &  &   \sigma^{n-1}(x_0)
  \end{bmatrix}
  \begin{bmatrix} g_0 \\ g_1 \\  \vdots\\ \\  \\  \\ g_{n-1} \end{bmatrix},
  \end{align}
 whereas multiplying by $b$ on the right looks like
  \begin{align}
    g b  = &  (g_0 + g_1 b + g_2 b^{2} + \dots + g_{n-1} b^{n-1}) b  \\ \\
    = & \begin{bmatrix}    &   &    &    &   & \  &   & \gamma \\
    1 &   &   &   &   &   &       &    \\
      & 1 &   &   &    &   &      &   \\
      &   & 1 &   &   &   &   &    \\
      &   &   & 1 &   &   &   &    \\
    \  &   &   &   &   & \ddots  &  &    \\
      &   &   &   &   &  & 1 &    
  \end{bmatrix}
  \begin{bmatrix} g_0 \\ g_1 \\  \vdots\\ \\  \\  \\ g_{n-1} \end{bmatrix}.
  \end{align}

  Extending this to the right multiplication by some $y = y_0 + y_1 b + \dots  + y_{n-1}{b^{n-1}}$, we write that 
  \begin{align}
    g y  = &    g (y_0 + y_1 b + y_2 b^{2} + \dots + y_n b^{n-1} ) \\ =  
    &  g( y_0 + b \sigma^{-1}(y_{1}) + b^{2} \sigma^{-2}(y_2) + b^{3} \sigma^{-3}(y_{3})+ \dots + b^{n-1} \sigma^{-n+1}(y_{n-1}) ) \\ \\ 
    =  & \begin{bmatrix}  
      y_{0}  			& \gamma \sigma(y_{n-1})  &  \gamma \sigma^{2}(y_{n-2})  &  \gamma \sigma^{3}(y_{n-3})   &  \  & \gamma \sigma^{n-2}(y_{2})   & \gamma \sigma^{n-1}(y_1)\\
      y_{1} 		&\sigma(y_{0}) 	& \gamma \sigma^{2}(y_{n-1})  &  \gamma \sigma ^{3}(y_{n-2})  &  &  \gamma \sigma^{n-2}(y_{3})     & \gamma \sigma^{n-1}(y_{2})    \\
      y_{2}		& \sigma(y_{1}) 	&  \sigma^{2}( y_{0}) 		&  \gamma \sigma^{3}(y_{n-1})  &  \dots &  \gamma \sigma^{n-2}(y_{4})  &  \gamma \sigma^{n-1}(y_{3}) \\
      y_{3}		&  \sigma(y_{2})  & \sigma^{2}(y_{1}) & \sigma^{3}(y_{0})  	&   & \gamma \sigma^{n-2}(y_{5}) & \gamma \sigma^{n-1}(y_{4})    \\
      y_{4} 	&  \sigma(y_{3}) &   \sigma^{2}(y_2)& \sigma^{3}(y_1)  &   &  \gamma \sigma^{n-2}(y_{6})  &  \gamma \sigma^{n-1}(y_{5})  \\
    \  				&\vdots   &   &   &    \ddots  &  &  &    \\
    y_{n-1} & \sigma(y_{n-2})   & \sigma^{2}(y_{n-3})  & \sigma^{3}(y_{n-4}) &  & \sigma^{n-2}(y_1) &    \sigma^{n-1}(y_{0})
  \end{bmatrix}
  \begin{bmatrix} g_0 \\ g_1 \\  \vdots\\ \\  \\  \\ g_{n-1} \end{bmatrix}.
  \end{align}
  Since this is a matrix representation of the right multiplication, we get from the above matrix a map $D^{\opp} \rightarrow M_{n}(E)$. 
  
  Clearly, $F$ lies in the center $\mathcal{Z}(D)$. In fact, after some matrix computations, one can see that $F$ is the center. From the identification (\ref{eq:iDentify}), it is clear that $\dim_{F}(D) = n^{2}$.
  
  \begin{remark} \label{re:nonnorm}
  If only the first three condition are satisfied in the definition without the condition \ref{it:nonnorm}, then we simply call $D$ a cyclic $\mathbb{Q}$-algebra. A cyclic $\mathbb{Q}$-algebra is a division algebra if and only if \ref{it:nonnorm} is satisfied. That is $(E,F,\sigma,\gamma)$ is a division algebra if and only if $\gamma$ is a non-norm element.
  \end{remark}

    \subsection{Amitsur's results}

    The problem of finding groups that can be embedded in division algebras was completely solved by Amitsur in his work \cite{Amit55}. Here we summarize the findings therein.
    

   Consider the following notation. 
   \begin{itemize}
     \item $m,r \in \mathbb{N}$ are two coprime integers.
     \item $n = \ord_{m} r$ is the multiplicative order or $r$ modulo $m$, that is the smallest positive integer $k$ such that $m  \mid r^{k}-1$.
       \item $s = \gcd(r-1,m)$. 
       \item $t = m/s$.
   \end{itemize}

   When $r=1$, we will assume $n=s=1$. In these definitions, we will think of $m$ and $r$ as two parameters and $n,s,t$ will automatically be set as defined above.

   With this, consider the cyclic algebra $\mathfrak{U}_{m,r} = (\mathbb{Q}(\mu_{m}), F , \sigma_{r}, \mu_{m}^{t} )$, where $F$ is the subfield of $\mathbb{Q}(\mu_{m})$ fixed by $\sigma_{r}$, and $\sigma_{r}$ is the field automorphism of $\mathbb{Q}(\mu_{m})$ given by $\mu_{m} \mapsto \mu_{m}^{r}$. 
   A priori, $\mathfrak{U}_{m,r}$ is just a $\mathbb{Q}$-algebra which may not be a division algebra. For this to be a division algebra, we want that $\mu_{m}^{t}$ is a non-norm element of $F$. 
   
   We can find out the dimension of $\mathfrak{U}_{m,r}$ as follows, $\dim_{\mathbb{Q}} \mathfrak{U}_{m,r} = n \dim_{\mathbb{Q}} E = n \varphi(m)=\varphi(m) \ord_{m} r $.

   Define $G_{m,r}$ to be the group given as 
   \begin{align}
     \langle A,B \ | \ A^{m} = 1, B^{n} = A^{t}, BAB^{-1} = A^{r}\rangle.
   \end{align}
   We get that $\# G_{m,r} = mn = m \ord_{m}r$. When $r=1$, $G_{m,r}$ is a cyclic group of order $m$.

   Consider the map $i:G_{m,r} \rightarrow \mathfrak{U}_{m,r}^{*}$ defined sending $A \mapsto \mu_{m}$ and $B \mapsto b$ (recall, $b \in \mathfrak{U}_{m,r}$ was a formal element such that Equation (\ref{eq:iDentify}) holds. Using Equation (\ref{eq:ruleofD}), we can conclude that this is a group homomorphism. It is injective, whether or not $\mathfrak{U}_{m,r}$ is a division algebra. 
   
   To find out whether or non $\mathfrak{U}_{m,r}$ is a division algebra amounts to checking whether or not $\gamma = \mu_{m}^{t}$ is a non-norm element, as mentioned in Remark \ref{re:nonnorm}. This can be done through the use o
   Hasse's local-global principles on the cyclotomic field $\mathbb{Q}(\mu_m)$; an element $\gamma \in F^{*}$ is a norm globally if and only if it always a norm locally. Doing this would yield some equivalent conditions on the numbers $m,r$. The following is Theorem 4 from \cite{Amit55} obtained from this method, stated here after being combined with Lemma 10 from that paper. Below, the notation $\ord_{a}b$ means the smallest positive power $k$ such that $a^{k} \equiv 1 \pmod{b}$.
   \begin{theorem} {\bf (Amitsur, 1955)}\\
     Consider the following conditions on the numbers $m,r,n,s,t$ defined above. Then $\mathfrak{U}_{m,r}$ is a division algebra if and only if both \ref{it:or1} and \ref{it:or2} given below hold.
     \begin{enumerate}
       \item \label{it:or1}
	 One of the following two conditions hold.
	 \begin{enumerate}
	   \item $\gcd(n,t) = 1$. This implies that $\gcd(s,t) = 1$.\label{it:or1co1}
	   \item
	     \label{it:or1co2}
	     $n=2n', m=2^{\alpha} m',s=2s'$, for some $\alpha \ge 2$ and $m',s',n'$ are odd numbers, such that $\gcd(n,t)=\gcd(s,t)=2$ and $ 2 ^{\alpha} \mid (r + 1) $. 
	 \end{enumerate}
       \item 
	 One of the following two conditions hold.
	     \label{it:or2}
	 \begin{enumerate}
	   \item $n=s=2$ and $ m \mid (r +1)$. \label{it:or2co1}
	   \item For every prime $ q \mid n$ there exists a prime $p \mid m$ such that if $m = p^{\alpha} m'$ with $p \nmid m'$, we get $q \nmid \ord_{m'} r $. In addition, at least one of the following must hold regarding $p,q$.
	     \begin{enumerate}
	       \item $p \neq 2$ and $\gcd(q,\frac{ p^{ \delta} - 1}{s} ) = 1$, where $\delta = \ord_{m'} p$.
	       \item $p=q=2$ and $m/4 \equiv \delta \equiv 1 \pmod{2}$, where $\delta$ is as above. This condition implies that the condition \ref{it:or1co2} above must hold.
	     \end{enumerate}
	 \label{it:or2co2}
	 \end{enumerate}
     \end{enumerate}\label{th:amit}
   \end{theorem}
   \begin{remark}
     About condition $\ref{it:or2co2}$ above, note that for a given prime $q \mid n$, there can exist at most one prime $p \mid m$ such that $q \nmid \ord_{ mp^{-\alpha}} (r)$, $\alpha$ being the power of $p$ in $m$. This is because if $m=p_1^{\alpha_1}p_2^{\alpha_2} \dots p_{k}^{\alpha_k}$ is the prime factorization of $m$, then 
     \begin{align}
       n=\ord_{m} r &  = \lcm \left(\ord_{m p_i^{-\alpha_i}}r , \ord_{p_i^{\alpha_i}} r \right) =  \lcm \left( \ord_{p_1^{\alpha_1}} r, \ord_{p_2^{\alpha_2}} r ,\dots, \ord_{p_k^{\alpha_k}}r \right),\\
       \ord_{mp_{i}^{ - \alpha_i}} r  & = \lcm \left( \ord_{p_1^{\alpha_1}} r,\dots, \ord_{p_{i-1}^{\alpha_{i-1}}} r ,   \ord_{p_{i+1}^{\alpha_{i+1}}} r ,\dots, \ord_{p_k^{\alpha_k}}r \right). 
   \end{align}
   So if $q \mid n$ but $q \nmid \ord_{m p_i^{-\alpha_i}}$ then $q\nmid \ord_{p_j ^{ \alpha_{j}}} r$ for each $j \neq i$ otherwise it would divide their $\lcm$. But then $q \mid \lcm \left(  \ord_{m p_j ^{-\alpha_j}} r, \ord_{p_j ^{\alpha_j}} r \right)$ so $q | \ord_{ m p_{j}^{-\alpha_j}  } r$.

   Hence, the prime $p$ whose existence is demanded in condition \ref{it:or2} exists uniquely depending on $q \mid n$.
   \end{remark}

   From Theorem \ref{th:amit}, we get a large family of $\mathbb{Q}$-division algebras $\mathfrak{U}_{m,r}$ and finite groups $G_{m,r}$ that embed inside them. When $m$ is odd and $\ord_{m} 2$ is odd, we can do slightly better and embed a group of size $24 | G_{m,r}|$ inside $\mathfrak{U}_{2,1} \otimes_{\mathbb{Q}} \mathfrak{U}_{m,r}$, which also is a division algebra. The next theorem says that apart from two more sporadic examples, these are all the finite groups that could concern us.

   The following is Theorem 7 from \cite{Amit55}.

   \begin{theorem} {\bf (Amitsur, 1955)}\\\label{th:amit_div}
     The following is an exhaustive list of finite groups $G_{0}$ that can be embedded in some $\mathbb{Q}$-division algebra $D$.
   \begin{center}
    \label{tab:table1}
    \begin{tabular}{l|l|l|l} 
      \textbf{Group} & \textbf{Conditions on the parameters} & \textbf{Size of} & \textbf{Dimension of the}\\ 
      \textbf{structure} & & \textbf{the group} &  \textbf{smallest division algebra} \\
      & & &  \textbf{containing the group} \\
      $G_{0} \subseteq D^{*}$ &  & $\#G_{0}$ &  $\dim_{\mathbb{Q}} \mathbb{Q}\langle G_{0}\rangle$  \\
      \hline
      \hline
      $\mathfrak{D}^{*}$&   & $48$  & $16$\\
\hline
      $\mathfrak{I}^{*}$&   & $120$  & $ 20 $\\
      \hline
      $G_{m,r}$ &  $r \le m$ are coprime and $\mathfrak{U}_{m,r}$ is a division algebra & $ m \ord_{m}r$ & $\varphi(m) \ord_{m}r$ \\
      \hline
      $ \mathfrak{T}^{*} \times G_{m,r}$ &  $r \le m$ are coprime and $\mathfrak{U}_{m,r}$ is a division algebra, & $24 m \ord_{m}r$ & $ 4 \varphi(m) \ord_{m}r$ \\
      & $m$ is odd and $\ord_{m} 2$ is odd. & & \\
      \hline
    \end{tabular}
  \end{center}
    Here $\mathfrak{T}^{*}, \mathfrak{D}^{*}, \mathfrak{I}^{*}$ are the binary tetrahedral group, binary octahedral group and binary icosahedral group respectively. They are finite groups whose respective size is 24,48 and 120.
\end{theorem}

\begin{remark}
  The claim that the stated dimension is that of the smallest division algebra that contains the group $G_{0}$ follows from Lemma 4 in \cite{Amit55} for the two infinite families.
\end{remark}

\section{Analysis and comparisons of bound obtained}\label{se:improv}

Recall the $c_d$ which we had defined as
\begin{align}
  c_{d} = \sup\left\{ \mu\left( g B_{r}(0)\right) \ | \ r> 0,~g \in SL_{d}(\mathbb{R}) \text{ and } g B_{r}(0) \cap \mathbb{Z}^{d} = \{ 0\}\right\}.
\end{align}

Theorem \ref{th:amit_div} along with Theorem \ref{pr:boundbetter} gives us the following result. Before that, let us briefly recall a theorem of Hasse \cite{Ha55}.

\begin{theorem}{\bf (Hasse, `66)}\\
  Define $\pi_{2}(x)$ as
  \begin{align}
    \pi_{2}(x)  & = \#\{ p \ | \ 2 < p \le x\text{ is prime and } p|(2^{m}+1)\text{ for some }m \in \mathbb{Z}_{\ge 0}\}\\
    & =\# \{ p \ | \ 2 < p \le x\text{ is prime and } \ord_{p}2\text{ is even} \}.
  \end{align}
  Then, we have that 
  \begin{align}
  \pi_{2}(x) = \frac{17}{24} \frac{x}{\log x} + o\left(\frac{x}{\log x}\right).
  \end{align}
  \label{th:hasse}
\end{theorem}

\begin{corollary}
  Using the prime number theorem, we get that if $\pi(x)$ is the prime-counting function, then the primes for which $\ord_{p} 2$ is odd follow the following growth.
  \begin{align}
    \pi(x) - \pi_2(x) = \frac{7}{24}\frac{x}{\log x} + o\left( \frac{x}{\log x} \right)
  \end{align}
\end{corollary}

\begin{theorem}
  There exists a sequence of dimensions $\{ d_i\}_{i=1}^{\infty}$ such that for some $C>0$, we have $c_{d_{i}} > C d_{i} (\log \log d_{i})^{\frac{7}{24}}$
  and the lattices that achieve this bound in each dimension are symmetric under the linear action of a non-commutative finite group.
  \label{th:improv}
\end{theorem}
\begin{proof}
  We pick $$m = \prod_{ \substack{ p \text{ is prime} \\ p \le x \\ 2 \nmid \ord_{p} 2} }^{} p$$ and $r=1$. Then observe that with this, we get that $m$ is odd and $\ord_{m}2$ is also odd. Using Theorem \ref{th:amit_div} and Theorem \ref{pr:boundbetter}, we get that $c_{8\varphi(m)} \ge 24m$.

  How do $\varphi(m)$ and $m$ grow with $x$? Define
  \begin{align}
    a_{n} = \begin{cases}
      1 \text{ if }n\text{ is an odd prime s.t.} 2 \nmid \ord_{p} 2\\
      0 \text{ otherwise}
    \end{cases}.
  \end{align}
  Then using Abel's summation formula and recalling $\pi_{2}(x)$ defined in Theorem \ref{th:hasse}, we get 
  \begin{align}
     \log \varphi(m) - \log m & =   \sum_{n=1}^{x } a_n \log\left(1-\frac{1}{n}\right) \\
     & = -\sum_{n=1}^{x} \frac{a_n}{n} + O(1)\\
     & = -\frac{ \pi(x) - \pi_2(x)  }{x} + O(1)  + \int_{1}^{x} \frac{\pi(t)-\pi_2(t)}{t^{2}} dt\\
     & = -\frac{7}{24} \log \log x  + o\left( \log \log x \right)\\
     & =  - \log \left( \log x \right)^{\frac{7}{24}} +  o \left( \log \log x \right).
  \end{align} 
  whereas
  \begin{align}
    \log m & = \sum_{n=1}^{x} a_n \log n \\
    & = \left( \pi(x) - \pi_2(x) \right)\log x - \int_{1}^{x} \frac{\pi(t) - \pi_{2}(t)}{t}dt \\
    & = \frac{7}{24}{x} + o\left({x} \right) \\
    \Rightarrow \log \log m & = \log x + o (\log x)\\
    \Rightarrow \log \log \varphi(m) & = \log x + o (\log x).
  \end{align}
  Putting this together, we get 
  \begin{align}
    \log m  & = \log \varphi(m) + \log (\log \log \varphi(m) )^{\frac{7}{24}} + o(\log \log \log \varphi(m)) \\
    \Rightarrow  m & > C \varphi(m) (\log \log \varphi(m))^{\frac{7}{24}} \text{ for some C > 0}
  \end{align}
\end{proof}

An analysis of the sequence of examples obtained through this has been done in Figure \ref{fig:graph}.

Another interesting sequence is the following. 
Let $m$ be any even number. Choose $r= m-1$. Then

We find that for this choice, 
\begin{align}
  n & = \ord_{m} r = \ord_{m} (-1) = 2, \\
  s & = \gcd(r-1,m)= \gcd(m-2,2)=2, \\
  t & = m/s = m/2  = \prod_{i=2}^{N} p_i .
\end{align}

Then we can check that the above choice of $(m,r,n,s,t)$ satisfies the conditions \ref{it:or1co1} and \ref{it:or2co1} of Theorem \ref{th:amit}.

\begin{proposition}\label{pr:asymptotic}
  Suppose $m = \prod_{i=1}^{N} p_i$, the product of first $N$ primes.  Then 
  \begin{align}
    c_{4 \phi(m)} \ge 2m,
  \end{align}
  and the lattice that achieves this bound is symmetric under the linear action of a non-commutative finite group. Along this sequence of dimensions $c_{d} \ge \frac{1}{2}(d \log \log d)$ eventually.
\end{proposition}

More exotic examples can also be constructed.
With $\{ p_1,p_2,\dots\}$ being the sequence of all primes, suppose $q= 1+\prod_{i=1}^{N}p_{i} $ is a prime for some $N$. Then we can choose an integer $r$ such that $r \equiv 1 \pmod{p_i}$ for each $i$ but has $\ord_{q} r = q-1$, i.e. $r$ is a generator of $\mathbb{F}_{q}^{*}$. Set $m=q(q-1) = q \prod_{i=1}^{N} p_i$. This gives us 
\begin{align}
  \ord_{m}r = \lcm \left( \ord_{q}r , \ord_{p_1} r, \dots, \ord_{p_i}r \right) = q-1.
\end{align}
Then we can check that this choice of $(m,r,n,s,t)$ satisfies Theorem \ref{th:amit}, conditions \ref{it:or1co1} and \ref{it:or2co2}.
\begin{align}
  (m,r,n,s,t)= \left(q(q-1), r, q-1, \gcd(r-1,q(q-1)) ,\tfrac{q(q-1)}{\gcd(r-1,q(q-1))}\right).
\end{align}

Since $\varphi(q(q-1)) = (q-1)\varphi(q-1)$, using Theorem \ref{th:amit_div} and Theorem \ref{th:improv} we get 
\begin{align}
  c_{2 (q-1)^{2}\varphi(q-1) } \ge q(q-1)^{2}.
\end{align}

Whether or not there are infinitely many primes of the form $ 1 + \prod_{i=1}^{N} p_{i} $ is a notorious open problem. Such primes are called primorial primes. 


Finally, it is worth pointing out that no sequence constructed using Theorem \ref{pr:boundbetter} can give us an asymptotic growth strictly better than $O(d \log \log d)$. Indeed, looking at Theorem \ref{th:amit_div}, we observe that $|G_0|/\dim_{\mathbb{Q}}D$ can at most be $3 m/\varphi(m)$ for some sequence of integers $m$. Hence, using the division algebra approach outlined here, the best lower bound that can be attained on $c_d/d$ will be at most $O(\log \log d)$.

\section*{Acknowledgements}
I would like to thank my advisor Prof. Maryna Viazovska for her helpful suggestions, ideas and suggesting directions. I also thank Matthew DeCourcy-Ireland and Vlad Serban for their useful comments and spotting some typographical mistakes.

This work was funded by the Swiss National Science Foundation (SNSF), Project funding (Div. I-III), "Optimal configurations in multidimensional spaces", 184927.
\begin{appendices}

\section{Matrices over real semisimple algebra}
\label{se:mat_over_real_ss_algebra}

Let us set up some introductory preliminaries about semisimple algebras over $\mathbb{R}$. This material is useful to describe the ``coarse'' fundamental domain introduced in Section \ref{se:reduction_theory}.

\subsection{Real semisimple algebras}
Throughout this text, we will use the word $k$-algebra when we actually mean an associative unital $k$-algebra. Our story begins with the following well-known result.

\begin{theorem} \textbf{ (Artin-Wedderburn)}
  \label{th:AW}
  Suppose $A$ is a semisimple algebra over a field $k$. Then for some finite-dimensional $k$-division algebras $D_1, D_2, \dots, D_r$ and natural numbers $n_1, \dots , n_r$, we get the isomorphism
  \begin{align}    
    A \simeq M_{n_1}(D_1) \oplus \dots \oplus M_{n_r}(D_r).\label{eq:semisimple}
  \end{align}
\end{theorem}

The right side of Equation (\ref{eq:semisimple}) is always semisimple for any choice of finitely many finite-dimensional $k$-division algebras. Thus, any reader who is not familiar with these objects could take the definition of semisimple $k$-algebras as the object on the right side.

\begin{theorem} \textbf{(Frobenius)} \\
  The only finite-dimensional $\mathbb{R}$-divison algebras (up to isomorphism) are $\mathbb{R}$, $\mathbb{C}$ and $\mathbb{H}$.
\end{theorem}

The three $\mathbb{R}$-division algebras all have a special ``conjugation'' involution that is compatible with the canonical inclusion $\mathbb{R} \hookrightarrow \mathbb{C} \hookrightarrow \mathbb{H}$. The map $\overline{(\ )}:\mathbb{H} \rightarrow \mathbb{H}$ given as $a + i b + j c +  k d \mapsto a - i b - j c - k d$ ($a,b,c,d \in \mathbb{R}$ and $i,j,k$ canonically span $\mathbb{H}$) satisfies that for any $x,y \in \mathbb{H}$ we have $\overline{x.y} = \overline{y}. \overline{x}$. When restricted to $\mathbb{C}$, this is the usual complex conjugation and when restricted to $\mathbb{R}$, this is the idenitity map. Another important property is that for any $ a + i b + j c + k d =x \in \mathbb{H}$, $\overline{x} x = a^{2} + b^{2} + c^{2} + d^{2} \in \mathbb{R}_{\ge 0}$. 

The two theorems stated above give rise to the following corollary. 

\begin{corollary}
  Any semisimple $\mathbb{R}$-algebra is isomorphic to one of products of matrix algebras over $\mathbb{R}$, $\mathbb{C}$ and $\mathbb{H}$.
\end{corollary}

Matrix algebras over $\mathbb{R}$, $\mathbb{C}$ and $\mathbb{H}$ are well understood. One important property is that the conjugation map defined above can be extended to a ``conjugate transpose'' involution on such matrices by simply defining the mapping $[x_{ij}]^{*} = [\overline{x_{ji}}]$. With this, we can also define a positive definite quadratic form on these matrix algebras by sending $a \mapsto \Tr( a^{*} a)$. 

On a given finite-dimensional algebra over $\mathbb{R}$, it is possible to define the trace map $\Tr_{A}: A \rightarrow \mathbb{R}$ and the norm map $\N_{A}: A \rightarrow \mathbb{R}$ as the trace and the determinant of the matrix of the left-multiplication operation induced by any element.
Similarly, it is also possible to generalize the above involution simply by taking direct sums of the respective involutions for matrix rings over $\mathbb{R}, \mathbb{C}$ or $\mathbb{H}$. We will omit the subscripts in $\Tr_{A}$ and $\N_{A}$ when $A$ is clear from the context.

\begin{corollary}
  Any semisimple $\mathbb{R}$-algebra $A$ admits an involution $(\ )^{*} : A \rightarrow A$ such that the following conditions are satisfied.
  \begin{itemize}
    \item For any $a,b \in A$, we have $(a b) ^{*}  = b^{*} a^{*}$.
    \item $a \mapsto \Tr(a^{*} a  )$ is a positive definite quadratic form on $A$. i.e. it is always non-negative and is zero only when $a =0$.
  \end{itemize}
  \label{co:involution}
\end{corollary}
\begin{proof}
  Simply take the direct sum of the ``conjugate transpose'' operation defined above on each matrix component of the semisimple algebra $A$. It is then to be seen that the trace function on $A$ is a sum of traces on the right side of Equation (\ref{eq:semisimple}), when they are realized as real matrix algebras. For instance, we must see $M_1(\mathbb{C})$ as a $2$-dimensional matrix algebra under the mapping $a+ib \mapsto \left[ \begin{smallmatrix} a & -b \\ b & a\end{smallmatrix}\right]$. 
  
\end{proof}

\begin{definition}\label{de:pos_inv}
  Any involution $A \rightarrow A$ satisfying the two properties of Corollary \ref{co:involution} is said to be a positive involution on $A$.
\end{definition}

\begin{lemma}\label{le:enjoyable}
  Suppose $(\ )^{*}: A \rightarrow A$ is a positive involution. Then 
  \begin{itemize}
    \item $1_{A}^{*} = 1_{A}$.
    \item If $u \in A$ is a zero non-divisor\footnote{In a finite-dimensional algebra over a field $k$, being a zero non-divisor is equivalent to being a unit and is also equivalent to the left/right multiplication map being full-rank.}, then $(u^{*})^{-1} = (u^{-1})^{*}$.
    \item For $u \in A$, $\Tr(u) = \Tr(u^{*})$.
    \item The inner product induced by the positive definite quadratic form $x \mapsto \Tr(x^{*} x)$ is $\langle x,y\rangle   = \Tr(x^{*} y)$.
  \end{itemize}
\end{lemma}
\begin{proof}
  The proofs are very enjoyable, so we leave all of them for the reader. The third one will require the use of semisimplicity of $A$, which implies that the left-multiplication trace and right-multiplication trace are the same.
\end{proof}

The notions of symmetric and positive definiteness can also be defined for $(A,(\ ) ^{*})$. 

\begin{definition}
  Given a finite-dimensional semisimple $\mathbb{R}$-algebra and an involution $(\ )^{*}$ as mentioned in Corollary \ref{co:involution}, we shall call an element $a \in A$
  \begin{itemize}
    \item symmetric, if $a ^{*} = a$.
    \item positive definite, if $x \mapsto \Tr(x^{*} a x)$ is a positive definite quadratic form on $A$.
  \end{itemize}
\end{definition}
\begin{lemma}\label{le:nondiv}
\begin{itemize}
\item For any unit $a \in A$, $a^{*} a $ is always symmetric and positive definite. 
\item If $a \in A$ is positive definite then $a$ is a zero non-divisor and $\Tr(a) > 0$.
\end{itemize}
\end{lemma}
\begin{proof}
  The first is a trivial verification. 
  
  For the second, note that if $a$ is a zero divisor then there exists some non-zero $x \in A$ such that $ax=0 \Rightarrow \Tr(x^{*} a x) = 0$ which contradicts the positive definiteness of $a$. Finally $\Tr(a) = \Tr(1^{*}_{A} a 1_{A}) > 0$.
\end{proof}


The above notions give us an opportunity to describe the following folklore lemma. It is often called the ``norm-trace'' inequality. It is used multiple times in the proof of Theorem \ref{th:MS} in \cite{W58} and appears to be central to the reduction theory of Section \ref{se:reduction_theory}. We state the lemma here for completion, even though the usage has been hidden away by citing Weil.

\begin{lemma}\label{le:normtrace}
  Consider a finite-dimensional semisimple $\mathbb{R}$-algebra $A$ with a positive involution $(\ )^{*}$. Let $a \in A$ be a symmetric positive definite element and let $d = \dim_{\mathbb{R}} A$. Then $\N(a) > 0$, $\Tr(a) > 0$ and 
  \begin{align}
    \frac{1}{d}\Tr(a) \ge  \N(a)^{\frac{1}{d}}.
  \end{align}
\end{lemma}
\begin{proof}
  This is just the arithmetic-geometric means inequality. Let us elaborate how.

  We know that $x \mapsto \Tr(x^* y)$ is an inner product on $A$. With respect to this, construct an orthonormal basis $e_1, e_2 ,\dots, e_d$. Set $a_{ij} = \Tr(e_i^{*} a e_{j})$ which are the matrix entries of left-multiplication by $a$ with respect to the basis $\{ e_{i}\}_{i=1}^{d}$, i.e. for $\{ r_i\}_{i=1}^{d} \subseteq \mathbb{R}^{d}$, $a ( \sum_{i} r_i e_i) = \sum_{i}^{} ( \sum_{j} a_{ij}r_j) e_i$. Since $a$ is symmetric, we get that $a_{ij} = a_{ji}$. Furthermore, by the positive definiteness of $x \mapsto \Tr(x^{*} a x)$, the matrix $a_{ij}$ can be seen to be positive definite as a real matrix by substituting $x = \sum_{i=1}^{d} x_i e_i$.

  Hence using the spectral theorem for real positive definite symmetric matrices, $a_{ij}$ is diagonalizable matrix with respect to an orthonormal change of basis and has real and positive eigenvalues (i.e. the diagonal entries). Then trace is the sum of those eigenvalues and the norm is the product. The inequality is then exactly the arithmetic-geometric inequality on those eigenvalues.
\end{proof}

\subsection{Cholesky decomposition}
\label{se:cholesky}

Let $A$ be a semisimple $\mathbb{R}$-algebra with a positive involution $(\ )^*$. The algebra $M_k(A)$ is also a semisimple $\mathbb{R}$-algebra and the involution $(\ )^{*}$ can be easily extended to $M_k(A)$ via the mapping $[a_{ij}] \mapsto [a_{ji}^{*}]$. We will denote this involution with same notation $(\ )^{*}$. With this, the meaning of positive definite and symmetric matrices in $M_k(A)$ is unambiguous. For clarity, we will distinguish between the norms and traces of $A$ and $M_k(A)$ by using the notations $\Tr_{A}, \N_{A}, \Tr_{M_k(A)}, \N_{M_k(A)}$ whenever appropriate.

For any $a \in M_k(A)$, we can create a bilinear form on $$\beta_{a}:A^{k} \times A^{k} \rightarrow \mathbb{R}$$ as $\beta_{a}(x,y) = \sum_{i,j=1}^{k} \Tr_{A}( x_{i}^{*} a_{ij} y_{j})$. 
The following lemma then approves that the conventional intuition of positive definiteness is in confirmation with the definition above.
\begin{lemma}
  An element $a \in M_k(A)$ is positive definite if and only if $\beta_{a}$ is a positive definite quadratic form on $A^{k}$ as an $\mathbb{R}$-vector space.
\end{lemma}
\begin{proof}
  We leave this for the reader.
\end{proof}

This lemma leads to the following decomposition for quadratic forms $\beta_{a}$ induced by symmetric positive definite matrices $a$.
What the upcoming theorem is really going to tell us is that the quadratic form $\beta_{a}$ can be ``diagonalized'' up to a ``triangular'' change of basis.

When $A =\mathbb{R}$, this is simply the Cholesky decomposition of real symmetric positive definite matrices. In \cite{W58}, the theorem below is referred to as the Babylonian reduction theorem, perhaps because it is spiritually similar to ``completing the square'' in a quadratic equation of one variable. 
\begin{theorem}\label{th:cholesky}
  Let $a \in M_k(A)$ be a symmetric positive definite matrix. Then there is an upper triangular matrix $t \in M_k(A)$ with $1_{A}$ on the diagonal entries, and a diagonal matrix $d$ with symmetric positive definite elements of $A$ on the diagonal such that 
  \begin{align}
    a = t^{*} d t.
  \end{align}
  That is, writing explicitly in terms of $A$-valued matrix entries, we can find $d,t \in M_{k}(A)$ such that 
  \begin{align}
    \begin{bmatrix} 
      a_{11}  & a_{12} &  & a_{1k} \\
      a_{21}  & a_{22} &  & a_{2k} \\
        &                & \ddots &  \\
      a_{k1}  & a_{k2} &  & a_{kk} 
    \end{bmatrix}
    = & 
    \begin{bmatrix} 
      1_A  &        &  &        \\
      t_{12}^{*}  & 1_A &  &        \\
            & \vdots               & \ddots &  \\
	    t_{1k}^{*}  & t_{2k}^{*} &  & 1_A      
    \end{bmatrix}
    \begin{bmatrix} 
      d_{11}  &        &  &        \\
              & d_{22} &  &        \\
        &                & \ddots &  \\
              &        &  & d_{kk} 
    \end{bmatrix}
    \begin{bmatrix} 
      1_A  & t_{12} & \dots & t_{1k} \\
              & 1_A      &  & t_{2k} \\
        &                & \ddots &  \\
              &        &  & 1_A      
    \end{bmatrix}\\
     = & \label{eq:decomp}
    \begin{bmatrix} 
      d_{11}  & d _{11}t_{12} & \dots & d _{11} t_{1k} \\
      t_{12}^{*} d_{11}        & t_{21}^{*}d_{11} t_{12} + d_{22}      &  &  t_{12}^{*} d_{11} t_{1k} +  d_{22}t_{2k} \\
     \vdots &                & \ddots &  \\
     t_{1k}^{*} d_{11}       &        &  & \sum_{i=1}^{k} t_{ki}^{*} d_{ii} t_{ik}
    \end{bmatrix}.\\
  \end{align}
\end{theorem}
\begin{proof}
  See \cite[Theorem 1]{W58}

\end{proof}


\begin{remark}
  The decomposition above is unique, because the elements $d_{ii}$ and $t_{ij}$ are completely determined by Equation (\ref{eq:decomp}).
\end{remark}
\begin{remark}
  It is possible to view $A^{k}$  as a $(k \dim_{\mathbb{R}}A)$-dimensional vector space over $\mathbb{R}$ and all the matrices in $M_k(A)$ can be seen as block matrices with each entry $a_{ij}$ being replaced by its left-multiplication matrix as an element of $A$. From this point of view, Theorem \ref{th:cholesky} is the same thing as the block matrix variant of the Cholesky decomposition.
\end{remark}

There is a further improvement that is possible to be done here using the proposition below.

\begin{proposition}
  Suppose that $a \in A$ is a positive definite symmetric element. Then, there exists another positive definite symmetric element $b \in A$ such that $b^{2} = a$.
  \label{pr:squaring}
\end{proposition}
\begin{proof}
  See \cite[Lemma 9.5]{W58}.


\end{proof}

\begin{corollary} \label{co:squareroot}
  For every positive definite symmetric element $a \in A$, $a = b^{*}b$ for some positive defining symmetric $b \in A$.

  Every element $a \in M_k(A)$ that is positive definite can be written in the form of
  \begin{align}
    a  = t^{*} b^{*} b t = p^{*}p,
  \end{align}
  where $t \in M_k(A)$ is upper triangular with $1_{A}$ on the diagonal, $b \in M_k(A)$ is diagonal and $p \in M_k(A)$ is just upper triangular.
\end{corollary}
\begin{proof}
  Use Theorem \ref{th:cholesky} and decompose $a$ as $t^{*}dt$. Then each diagonal entry of $d$ can be split as $d_{ii} = b_{ii}^{*}b_{ii}$ according to the previous corollary.
\end{proof}

\section{ Setting up the Haar measure on $G$}
\label{se:haar_on_G}

This section is adding to the description of our choice on Haar measure on $G = SL_k(D_{\mathbb{R}})$. The main purpose is to fill in the missing details about $G$ in Section \ref{se:reduction_theory}.

\begin{proposition} \label{pr:iwasawa}
  The following map is a surjective open map. As a smooth map, it is a submersion.
\begin{align}
  K \times A_0 \times N & \rightarrow G \\
  (\kappa,a,n) & \mapsto \kappa a n.
\end{align}
\end{proposition}
\begin{proof}
  First, let us see that this multiplication map is surjective. 

  For any $g \in G$, we know that $g^{*} g \in M_k(A)$ is a positive-definite symmetric matrix. Consequently, by Theorem \ref{th:cholesky} and Corollary \ref{co:squareroot}, we have a decomposition $g^*g = n^{*} a^{*} a n = (an)^{*} a n$, for $n \in N$ and $a$ being some diagonal matrix. Clearly $N(a)= \pm 1$ for this to hold, but since $a$ can be assumed to be positive definite in Corollary \ref{co:squareroot}, we can ensure that $N(a_{ii}) > 0 $ and so $a \in A_{0}$.
  Now $g(an)^{-1} =  (g^{*})^{-1}(an)^{*}$ which means that $g(an)^{-1}$ is preserved under the ``conjugate inverse'' automorphism, so it lies in $K$. So $g = \kappa an $ for some $\kappa \in K$.
  

  Using the following transportation scheme, we can see that the given map has a constant rank. The following commutative diagram demonstrates that the rank at $(\kappa,a,n)$ is the same as the rank at $(e,e,e)$, wherein the vertical arrows are the derivatives of the respective indicated maps and are therefore isomorphisms of tangent spaces.

\[\xymatrixcolsep{4pc}\xymatrix{
    \left(\kappa \kappa',a'a, ( a^{-1}n'a ) n \right) & T_{(\kappa,a,n)}(K \times A_{0} \times N ) \ar[r] & T_{\kappa an }G & \kappa g an  \\
    (\kappa',a',n')\ar@{|->}[u]  & T_{(e,e,e)}( K \times A_{0} \times N ) \ar[r] \ar[u] & T_{e}G \ar[u] & g \ar@{|->}[u]\\
  }
\]

To learn the rank on $(e,e,e)$, we note that on the level of Lie algebras the lower horizontal map in the diagram, up to appropriate identifications, is just the addition map. More precisely, we can make the identification of $T_{(e,e,e)}( K \times A_{0} \times N ) \simeq T_{e}K \times T_{e} A_{0} \times T_{e} N $ and identifying $T_{e}G,T_{e} K , T_{e} A_{0}$ and $T_{e} N $ as subspaces of $T_{e} GL_{k}(D_{\mathbb{R}}) \simeq M_k(D_{\mathbb{R}})$ as follows.

\begin{align}
  T_{e}G&  =   \{ g \in  M_k(D_{\mathbb{R}}) \ | \ \Tr(g) = 0 \},\\ 
  T_{e}K &  = \{ \kappa \in M_k(D_{\mathbb{R}}) \ | \ \kappa^{*} +  \kappa = 0  \}, \\
  T_{e} A_{0} & = \{ a \in M_k(D_{\mathbb{R}}) \ | \ a \text{ is diagonal}, \Tr(a) = 0\}, \\
  T_{e} N & =  \{ n \in M_k(D_{\mathbb{R}}) \ | \ n \text{ is strictly upper triangular }\}.
\end{align}

Since every traceless matrix in $M_k(A)$ can be written as the sum of matrices in the three subspaces above, the bottom map is surjective and hence overall, the given map is a submersion using the global rank theorem of differential geometry.

To see that it is an open map, it is sufficient to show that the image $U_1U_2 U_3 \subseteq G$ of a basic open set of the product topology $U_1 \times U_{2} \times U_{3} \subseteq K \times A_{0} \times N$ is open\footnote{for any continuous map of topological spaces $f:X\rightarrow Y$, $f(\bigcup_{i \in I} U_i) = \bigcup_{i \in I} f(U_{i})$.}. For this goal, it is sufficient to show this when $U_1  \times U_{2} \times U_{2}$ is a sufficiently small neighbourhood of the identity $(e,e,e)\in K \times A_{0} \times N $, as we can transport such a neighbourhood and get a neighbourhood $(\kappa,a,n)\in K \times A_{0} \times N$ of the form $\kappa U_{1} \times U_{2} a \times  ( a^{-1} U_{3} a ) n  $, whose image must be $\kappa U_1 U_2 U_3 an \subseteq GL_k(A)$. 
For this, it is also sufficient to show that the given multiplication map restricted to $U_1 \times U_2 \times U_3$ is an open map for sufficiently small $U_1,U_2,U_3$.

We use the constant rank theorem of differential geometry to do this. 
Let $F: K\times A_{0} \times N \rightarrow G$ be the given multiplication map. If $U_1, U_2, U_3$ are sufficiently small, then there exists an open neighbourhood $U'_1 \times U_{2}' \times U'_{3}\subseteq T_{e} K \times T_{e} A_{0} \times T_{e} N $ of $(0,0,0)$ with homeomorphisms $u_i:U_{i} \rightarrow U'_i$ for $i \in \{ 1,2,3\}$ and an open neighbourhood $V \subseteq G$ containing identity along with a homeomorphism $v : V \rightarrow V' \subseteq T_{e} G$, $V'$ containing $0$, such that $F|_{U_1 \times U_2 \times U_{3}} = v^{-1} \circ dF_{(e,e,e)} \circ u$, where the map $u = u_1 \times u_2 \times u_3 : U_1 \times U_2 \times U_3 \rightarrow U'_1 \times U'_2 \times U'_3$. But $dF_{(e,e,e)}$ is an open map, because it is a surjective linear map and hence we are done.
\end{proof}

\begin{corollary}\label{co:kbopenmap}
  Let $B = A_{0} N = N A_{0} \subset G$ be the closed subgroup of upper-triangular matrices. Then the following is also an open surjective map.
  \begin{align}
    K \times B \rightarrow & G \\ 
    (\kappa, b) \mapsto & \kappa b.
  \end{align}
  \begin{proof}
    Surjectivity is clear from Proposition \ref{pr:iwasawa} if we write $b = an $ for some $n \in N $ and $a \in A_{0}$. To show that the map is open, 
   the proof is very similar to the Proposition \ref{pr:iwasawa} and we leave this to the reader for verification.
  \end{proof}

\end{corollary}

\begin{remark}
  The map in Proposition \ref{pr:iwasawa} is generally not injective. Indeed, if $\kappa' \in K\cap A_{0}$, then $(\kappa \kappa'^{-1} , \kappa' a, n )$ and $(\kappa,a,n)$ are mapped to the same element $\kappa an \in G$. 

  This is the only obstruction to injectivity. That's to say that, two elements of $K \times A_{0} \times N$ have the same image if and only if they are in the above situation. With the usual Iwasawa decomposition for $SL_k(\mathbb{R})$, the map is indeed injective since $K \cap A_{0} = \{ 1_{SL_k(\mathbb{R})}\}$.
\end{remark}

We will now use the following proposition to settle some more technicalities about our decomposition above. 

\begin{proposition}\label{pr:kcompact}
  \begin{enumerate}
    \item $K \subset G$ is a compact group.
    \item $K \cap  B= K \cap A_{0}$, which is also a compact subgroup of $G$.
  \end{enumerate}
\end{proposition}
\begin{proof}
  \begin{enumerate}
    \item $K$ is at most an index-2 subgroup of $\{ a \in M_k(D_{\mathbb{R}}), a^{*} a = 1_{M_k(\mathbb{R})}\}$. The compactness of this group follows from the following more general claim.

      Let $A$ be a semisimple algebra with a positive involution $^{*}$, then the group $\{ a  \in A \ | \ a^{*} a = 1_{A}\}$ must be a compact group in the induced topology from $A$. Indeed, it is a closed group that lives inside the compact ball $\{ a \in A \ |  \ \Tr_{A}(a^{*}a) \le [A:\mathbb{R}] \}$.
    \item We see that if $\kappa \in K \cap A_0N$ then $an$ has to be an upper triangular matrix such that $\kappa^{*} \kappa = 1_{G} $. As a matrix, what this means is that 
  \begin{align}
    \begin{bmatrix} 
      1_{D_{\mathbb{R}}}      &   &  &  \\
      & 1_{D_{\mathbb{R}}} &  &  \\
        &                & \ddots &  \\
      &        &  & 1_{D_{\mathbb{R}}}    \end{bmatrix}
    = & 
    \begin{bmatrix} 
      \kappa_{11}^{*}  &        &  &        \\
      \kappa_{12}^{*}  & \kappa_{22}^{*} &  &        \\
            & \vdots               & \ddots &  \\
	    \kappa_{1k}^{*}  & \kappa_{2k}^{*} &  & \kappa_{kk}^{*}      
    \end{bmatrix}
    \begin{bmatrix} 
      \kappa_{11}  & \kappa_{12} & \dots & \kappa_{1k} \\
      & \kappa_{22}      &  & \kappa_{2k} \\
        &                & \ddots &  \\
	&        &  & \kappa_{kk}      
      \end{bmatrix}\\
     = &
    \begin{bmatrix} 
      \kappa_{11}^{*} \kappa_{11}  & \kappa_{11}\kappa_{12} & \dots & \kappa_{11} \kappa_{1k} \\
      \kappa_{12}^{*} \kappa_{11}        & \kappa_{12}^{*} \kappa_{12} + \kappa_{22}^{*}\kappa_{22}      &  &  \kappa_{12}^{*} \kappa_{1k} +  \kappa_{22}\kappa_{2k} \\
     \vdots &                & \ddots &  \\
     \kappa_{1k}^{*} \kappa_{11}       &        &  & \sum_{i=1}^{k} \kappa_{ki}^{*} \kappa_{ik}
   \end{bmatrix}.\\
  \end{align}

  We will show that $\kappa_{ij}= 0$ for $i<j$. When $i=1$, we see that $ \kappa^{*}_{11}\kappa_{11} = 1_{A}$, so $\kappa_{11}$ is invertible and therefore from the first row above, we see that $\kappa_{11} \kappa_{1j} = 0 \Rightarrow \kappa_{1j} = 0 $ for $j > 1$. This makes the entire first row of $\kappa$, except $\kappa_{11}$ to be $0$. This reduces the case to a $(k-1) \times (k-1)$ upper triangular matrix satisying the same matrix equality as above. Hence, we can show the rest of the entries are $0$ by induction.

  Now $K\cap A_{0} \simeq \{ a \in D_{\mathbb{R}} \ | \ a^{*} a = 1_{D_{\mathbb{R}}} \}^{ \oplus k}$ as a topological group. From the discussion of the previous part, it is compact.
  \end{enumerate}
\end{proof}

One last piece of the puzzle describes something special about Haar measure on $G$.
\begin{proposition}
  The group $G$ is unimodular. That is, a left-invariant Haar measure is also right-invariant.
\end{proposition}
\begin{proof}
  First, observe that the group $GL_k(D_{\mathbb{R}})$ is unimodular. 
  
  $GL_k(D_{\mathbb{R}})$ is an open subset of $M_{k}(D_{\mathbb{R}})$. This is because for any $u \in GL_k(D_{\mathbb{R}})$ and $u' \in M_k(D_{\mathbb{R}})$, $u+tu' = u(1+tu^{-1}u')$ is invertible if $t\in \mathbb{R}$ satisfies $ |t|^{2}\Tr( (u^{-1} u')^{*} (u^{-1}u') ) < 1$. Hence, any Lebesgue measure $da$ of $M_{k}(D_{\mathbb{R}})$ can be restricted to get a measure $da$ on $GL_k(D_{\mathbb{R}})$. Now set $dg = | \N(a)|^{-1}da$. This measure is in fact both left and right invariant. Indeed, this is because the deterninant of the left-multiplication of $a \in M_k(D_{\mathbb{R}})$ is the same as that of the right-multiplication, both being equal to $| \N(a)|$.

  Now on $G=SL_k(D_{\mathbb{R}})$, we can induce a Haar measure as follows. For any open set $U \subseteq G$, consider the set $(0,1]U = \bigcup_{t \in (0,1]} tU \subseteq GL_k(D_{\mathbb{R}})$ and define $\mu_{G}(U) = \int_{(0,1]U} dg $. This defines a Haar measure on $G$ that is both left-invariant and right-invariant.
\end{proof}

All this machinery can now be used to show Proposition \ref{pr:Gisunimodular}. Here it goes.

\begin{proof} (of Proposition \ref{pr:Gisunimodular}).

    We use the following classically known lemma. See \cite{knapp2013}, for a proof.
    \begin{lemma}
      Let $G'$ be a Lie group. Let $S,T$ be closed subgroups such that $S \cap T $ is compact and the multiplication $S \times T \rightarrow G'$ is an open map whose image is surjective (except possibly a measure $0$ subset of $G'$). Let $\Delta_{T}$ and $\Delta_{G'}$ denote the modular functions of $T$ and $G'$. Then the following is a Haar measure on $G'$.
  \begin{align}
    C_{c}(G') \rightarrow &\  \mathbb{R} \\ 
    f \mapsto &  \  \int_{S \times T}f(s t)  \frac{\Delta_{T}(t) }{\Delta_{G'}(t) }  ds dt.
  \end{align}
    \end{lemma}

    We will use this lemma twice. First with $(G',S,T) = (G,K,B)$, which fits due to Corollary \ref{co:kbopenmap} and Proposition \ref{pr:kcompact}, and then $(G',S,T) = (B,A_{0},N)$ which fits because $A_{0} \cap N = \{ 1_{G}\}$ and $(a,n) \mapsto an$ is an open map. Then, we get that the following is a Haar integral for $f \in C_{c} (G)$.
    \begin{align}
      \int_{K \times B}^{} f(\kappa b) \frac{\Delta_{B}(b)}{\Delta_{G}(b)} d\kappa db & =  \int_{B}^{} \left( \int_{K}^{}f(\kappa b) \frac{\Delta_{B}(b)}{\Delta_{G}(b)} d\kappa \right) db \\
      & =  \int_{N}^{ } \int_{A_0}^{} \left( \int_{K}^{} f(\kappa an) \frac{\Delta_B{(an)}}{ \Delta_{G}(an)}  d\kappa\right) \frac{\Delta_{N}(n)}{\Delta_{B}(n)} da dn\\
      & =  \int_{N}^{ } \int_{A_0}^{}  \int_{K}^{} f(\kappa an) {\Delta_B{(a)}}  d\kappa  da dn.
    \end{align}
    Here, for the last equality we have used that that $G$ and $N$ are unimodular $\Delta_{G},\Delta_{N}$ are trivial. $G$ is unimodular by Lemma \ref{pr:Gisunimodular} and $N$ is unimodular because it is nilpotent\footnote{Alternatively, one can check this through the identity $\Delta_{N}(n) = | \det \Ad(n)|$ }. Finally, we use the following identity that is classically known and also given in \cite{knapp2013}.
    \begin{align}
      \Delta_{B}(a)  = |\det \Ad_{B} ( b)|,
    \end{align}
    where $\Ad_{B} : B \rightarrow GL(T_{e}B)$ is the adjoint representation of $B$. Identify
    \begin{align}
      T_{e} B = \{  m \in M_{k}(D_{\mathbb{R}}) \ | \ \Tr(m) = 0 , m \text{ is upper triangular}\}.
    \end{align}
    Then, clearly $(ama^{-1})_{ij} = a_{ii} n_{ij} a_{jj}^{-1}$. Since determinant of right multiplication and left multiplication on $D_{\mathbb{R}}$ is the same, we get 
    \begin{align}
      \Delta_{B}(a) = \prod_{i < j}^{} \left|\frac{N(a_{ii})}{N(a_{jj})}\right|.
    \end{align}
  \end{proof}

\end{appendices}

 \bibliographystyle{unsrt}
  \bibliography{auth}

\end{document}